\documentclass[11pt,a4paper]{article}
\usepackage[utf8]{inputenc}
\usepackage[T1]{fontenc}
\usepackage{amsmath,amssymb,amsthm,mathtools}
\usepackage{bm}
\usepackage{booktabs}
\usepackage{array}
\usepackage{graphicx}
\usepackage{enumerate}
\usepackage[margin=1in]{geometry}
\usepackage{xcolor}
\usepackage{hyperref}
\usepackage{microtype}
\usepackage{float}
\usepackage{comment}
\newfloat{algobox}{t}{loa}
\floatname{algobox}{Algorithm}

\hypersetup{colorlinks=true,linkcolor=blue!55!black,citecolor=green!45!black,urlcolor=blue!55!black}
\newcommand{\Cref}{\autoref}
\newcommand{\cref}{\autoref}
\AtBeginDocument{%
}

\theoremstyle{plain}
\newtheorem{theorem}{Theorem}[section]

\newtheorem{proposition}{Proposition}[section]
\newtheorem{corollary}{Corollary}[section]
\theoremstyle{definition}
\newtheorem{assumption}{Assumption}
\newtheorem{remark}{Remark}[section]

\newcommand{\Om}{\Omega}
\newcommand{\dx}{\,\mathrm{d}x}
\newcommand{\dt}{\Delta t}
\newcommand{\eps}{\varepsilon}
\newcommand{\epsb}{\eps_{b}}
\newcommand{\epsq}{\eps_{q}}
\newcommand{\grad}{\nabla}
\newcommand{\dv}{\grad\!\cdot\!}
\newcommand{\lap}{\Delta}
\newcommand{\pt}{\partial_t}
\newcommand{\inner}[2]{\bigl(#1,\,#2\bigr)}

\newcommand{\norm}[1]{\lVert #1 \rVert}

\newcommand{\R}{\mathbb{R}}
\newcommand{\str}{\star}
\newcommand{\bs}{b_\str}
\newcommand{\cs}{c_\str}
\newcommand{\As}{A_\str}
\newcommand{\ks}{\kappa_\str}
\newcommand{\Gs}{G_\str}
\newcommand{\Dtwo}{D_{2}}
\newcommand{\dd}{D^{2}}
\newcommand{\Ef}{E_{f}}
\newcommand{\Elin}{E_{\mathrm{lin}}}
\newcommand{\Emod}{\mathcal{E}}
\newcommand{\Dmod}{\widetilde{\mathcal{D}}}
\newcommand{\Dcont}{\mathcal{D}}
\newcommand{\Ered}{\widetilde{E}}
\newcommand{\Phimix}{\Phi}

\newcommand{\half}{\tfrac12}
\newcommand{\CO}{C_{0}}
\DeclareMathOperator{\clamp}{clamp}

\newcommand{\eqdef}{\vcentcolon=}
\usepackage{xspace}   
\newcommand{\eqS}{(S1)--(S4)\xspace}

\title{\bf A linear, second-order, unconditionally energy-stable\\
scalar-auxiliary-variable scheme for the\\
CHQ viscoelastic phase-separation model}
\author{
Ritesh Bhola\thanks{Institute of Physics, Johannes Gutenberg University
Mainz, Germany. \texttt{rbhola@uni-mainz.de}}
\and
Michael te Vrugt\thanks{Institute of Physics, Johannes Gutenberg University
Mainz, Germany. \texttt{tevrugtm@uni-mainz.de}}
}
\date{}
\begin{document}
\maketitle
\begin{abstract}
The CHQ model describes viscoelastic phase separation in a
polymer--solvent mixture: a Cahn--Hilliard equation for the polymer volume
fraction $\phi$ is two-way coupled, through a phase-dependent
cross-diffusion flux, to a damped relaxation equation for a scalar bulk
elastic stress $q$. We propose a scalar-auxiliary-variable (SAV) time
discretisation of this system, combining the SAV reformulation of the
double-well potential with a second-order backward differentiation
formula (BDF2) and a second-order extrapolation of all
phase-dependent mobilities: a linear, computationally cheaper complement
to the reference scheme's fully nonlinear Newton solve. The resulting method is (i)
linear at every time step, with a coupling structure we prove has
exactly one solution at every step, for every choice of time-step size;
(ii) mass conserving, exactly, as an algebraic identity of the
discrete equations themselves (the only drift ever seen numerically
traces to the finite tolerance of the iterative linear solver, not a
failure of the conservation property); and (iii) unconditionally stable for a modified
discrete energy, via an exact one-step energy--dissipation identity. We show the energy identity extends to this
filtered variant as well, with two explicit extra non-negative terms. Hence, the scheme actually producing the
physics results is itself provably
stable for the same energy. A Fourier--Galerkin implementation verifies that the energy identity holds at every time step to within floating-point round-off. 
It also confirms second-order temporal accuracy, including in the steep-interface regime, and agreement with an independently implemented nonlinear Newton--BDF2 solver. Compared with that independent nonlinear
Newton--BDF2 solver, at equal grid resolution, the SAV method uses
$2.5$--$3.0\times$ fewer linear-solver iterations and $1.7$--$2.0\times$
less wall-clock time.
\end{abstract}

\section{Introduction}

\subsection{The model}

Phase separation in polymer solutions and melts is frequently accompanied by
a pronounced viscoelastic response: the slower-relaxing polymer network
resists deformation on the timescale of demixing, and the resulting elastic
stress can dramatically alter, and in some cases arrest, the coarsening process
compared to a simple binary fluid. The same coupling between phase
separation and viscoelasticity is central to biomolecular condensates in
living cells \cite{Tanaka2022}. Condensates themselves age into
viscoelastic, Maxwell-fluid-like materials rather than remaining simple
liquids \cite{Jawerth2020}, and the surrounding chromatin network has been
shown experimentally to slow the coalescence of nucleolar droplets by orders
of magnitude relative to isolated liquid droplets, an arrest effect closely
analogous to the elastic suppression of coarsening studied here
\cite{Caragine2018}. More directly still, recent experiments show that
tuning the viscoelastic properties of a nucleus-mimicking compartment
controls the dynamics and liquid-to-solid transition of protein condensates
\cite{Fritzen2026}. This phenomenon, reviewed
by Tanaka \cite{Tanaka2000} and modelled by a hierarchy of
increasingly reduced continuum descriptions, couples a Cahn--Hilliard-type
phase-field equation to an evolution equation for the elastic stress through
a cross-diffusion flux. This structure is numerically delicate:
energy-stable schemes for the uncoupled Cahn--Hilliard equation do not
automatically extend to the coupled system, and the reference discretisation
for exactly this reduced model \cite{Brunk2025} requires a nonlinear solve
at every time step. This paper shows that a linear scheme can retain the
same unconditional energy stability.

The scalar auxiliary variable (SAV) approach
\cite{ShenXuYang2018,ShenXuYang2019} is a general strategy for exactly this
kind of question: introducing one extra scalar unknown that tracks the
square root of the (shifted) nonlinear part of the energy turns the
troublesome nonlinear term (the derivative $f'(\phi)$ of the double-well
potential in \eqref{eq:chq-mu} below) into a linear one, at the cost of
solving one additional, cheap scalar equation each step. Applied to a single
Cahn--Hilliard equation this already gives a linear, unconditionally
energy-stable scheme with no step-size restriction and no extra
stabilisation term added by hand, the usual efficiency case made for SAV
schemes. It is, however, nontrivial to extend this to the viscoelastic phase separation model given by \cref{eq:chq} 
below, where the SAV-linearised equation for the density $\phi$ and the equation for the stress $q$ are coupled through
the same cross-diffusion flux (\Cref{sec:positioning}).

Following Spiller et al.\ \cite{Spiller2021}, Zhou et al.\ \cite{ZhouZhangE2006}, and 
Brunk et al.\ \cite{Brunk2025}, we consider a model without hydrodynamics that has the polymer
volume fraction $\phi$, the chemical potential $\mu$, and a single scalar bulk
elastic variable $q$ as dynamical variables. We consider the following system on a periodic domain $\Om\subset\R^d$ for $t\in(0,T)$:
\begin{subequations}\label{eq:chq}
\begin{align}
\pt\phi &= \dv\!\bigl( b(\phi)\grad\mu \;-\; c(\phi)\grad(A(\phi)q) \bigr),
   \label{eq:chq-phi}\\[2pt]
\mu &= -\gamma\lap\phi + f'(\phi), \label{eq:chq-mu}\\[2pt]
\pt q &= -\kappa(\phi)\,q
   \;-\; A(\phi)\,\dv\!\bigl( c(\phi)\grad\mu - d_0\grad(A(\phi)q) \bigr)
   \;+\; \epsq\lap q, \label{eq:chq-q}
\end{align}
\end{subequations}

The analysis below is independent of the spatial dimension $d$, although our numerical experiments use $d=2$. We impose periodic boundary conditions to eliminate boundary terms in the integration-by-parts identities used in \Cref{sec:notation} and to match the Fourier spectral discretisation described in \Cref{sec:accuracy}. At the continuous level, other boundary conditions that eliminate the same boundary terms, such as suitable no-flux conditions, would also suffice. Here, $\gamma>0$ is the gradient-energy coefficient, $d_0,\epsq>0$ are constants, and $f$ is a double-well potential bounded from below. The functions $b,c,A,\kappa$ depend on the phase field $\phi$ and satisfy $b,\kappa>0$ and $c,A\ge0$.

\subsection{Energy structure}\label{sec:energy}

Equations \eqref{eq:chq} describe a system relaxing towards thermodynamic equilibrium:
$\phi$ and $q$ evolve so as to decrease a single scalar quantity $E$, with
the terms on the right of \eqref{eq:chq} acting as dissipative (frictional)
driving forces whose strength is set by the state-dependent coefficients
$b(\phi),c(\phi),A(\phi)$, together with, for $q$, additional linear
damping from the $-\kappa(\phi)q$ and $\epsq\lap q$ terms in
\eqref{eq:chq-q}. We define the energy as
\begin{equation}\label{eq:energy}
E(\phi,q)=\int_\Om\Bigl[\tfrac{\gamma}{2}|\grad\phi|^2+f(\phi)+\tfrac12 q^2\Bigr]\dx ,
\end{equation}
a sum of three physically distinct contributions: the interfacial energy
$\tfrac\gamma2|\grad\phi|^2$, which penalises sharp changes in composition
and so favours smooth, spread-out interfaces; the bulk free energy
$f(\phi)$, a double-well potential favouring the two pure phases
$\phi\approx0,1$ and so driving phase separation; and the elastic energy
$\tfrac12q^2$ stored in the polymer network, penalising deviations of the
elastic variable $q$ away from its relaxed value $q=0$.

For sufficiently smooth solutions that are periodic in space, the energy satisfies a dissipation law, ($\frac{\mathrm d}{\mathrm dt}E=-\Dcont$), as shown in \Cref{sec:cont-energy}. The dissipation rate is non-negative only under a suitable condition on the model coefficients. We derive this condition, called the coercivity condition, and establish the continuous energy law in \Cref{sec:cont-energy}.

\subsection{Related work}\label{sec:related}

Brunk, Egger, Habrich and Luk\'acov\'a-Medvidov\'a \cite{Brunk2025}
(Correction \cite{Brunk2025corr}) give a rigorously analysed, fully nonlinear
finite-element / variational-in-time discretisation of \eqref{eq:chq} with an
exact energy identity and order-optimal error estimates via relative energy,
and name the SAV route as future work for cheaper solves. Strasser, Tierra,
D\"unweg and Luk\'acov\'a-Medvidov\'a \cite{Strasser2019,LSDT2018} give
linear, energy-stable schemes for the full hydrodynamic polymer--solvent model
by a secant linearisation (not SAV), without an error analysis. Structural
cousins include a Cahn--Hilliard/Allen--Cahn cross-kinetic coupling
\cite{BEH2024}, viscoelastic Cahn--Hilliard models discretised by
compactness-driven (not SAV, not linear) schemes
\cite{GarckeKovacsTrautwein2022,GarckeTrautwein2024}, and the nearest SAV
scheme for any mechanically-coupled phase-field model, a Caginalp/elastic-solid
system \cite{JinLamYe2024} structurally distinct from \eqref{eq:chq}. The SAV
approach itself \cite{ShenXuYang2018,ShenXuYang2019,ShenXu2018} and the
frozen-mobility-at-extrapolated-state device \cite{VarMobCH2017} are
established for scalar gradient flows; single-auxiliary-variable SAV has been
used for other coupled Cahn--Hilliard systems \cite{CoupledSAV2023,TumourSAV2024}.
To the best of our knowledge, no SAV, IEQ, or energy-quadratisation
discretisation has previously been proposed for \eqref{eq:chq}, or for any
viscoelastic Cahn--Hilliard model with this cross-diffusion structure.

\subsection{Contribution}\label{sec:positioning}

For standard Cahn--Hilliard-type gradient flows, the SAV device separates the
nonlinear free-energy term from the remaining linear operator by freezing it
at an extrapolated state, and the resulting scheme's energy stability then
follows term by term. In \eqref{eq:chq}, however, $\phi$ and $q$ are coupled
through the same phase-dependent cross-diffusion flux
$c(\phi)\grad(A(\phi)q)$, entering the two equations with opposite sign:
freezing $c$ and $A$ independently, or at the new-time solution, breaks the
exact coupled-dissipation structure that produces the dissipation term
\eqref{eq:cont-D} in the continuous energy law, and reintroduces a
$\dt$-dependent remainder that would have to be absorbed by a stabilisation
term or a step-size restriction. The observation underlying \Cref{thm:main} below
is that freezing all of $b,c,A,\kappa$ at one common extrapolated
state $\phi^\star$ preserves this cross-diffusion structure exactly, so the
SAV reformulation and the BDF2 energy identity combine with no such remainder
and no restriction on $\dt$.

The new contribution of this work is not the SAV device or the extrapolation 
device in isolation, but to show that they
together survive the \emph{coupled} cross-diffusion dissipation structure of
\eqref{eq:chq}. Concretely:
\begin{enumerate}[(i)]
\item an SAV--BDF2 discretisation of the full coupled system, linear at every
step, with a proof (\Cref{thm:main}) that the modified-energy identity
\eqref{eq:star} extends exactly to the coupled dissipation form with
the $A(\phi)$-weighted test (holding for any time step $\dt>0$, with no
stabilisation term and no $\|\phi^{n+1}-\phi^\star\|$ remainder), together
with an extension (\Cref{rmk:filter}) to the post-step spectral-filtering
variant actually used for the physics study;
\item a per-step solvability theorem (\Cref{thm:solv}): the linear system
\eqS is uniquely solvable for every $\dt>0$;
\item numerical verification: the identity to round-off, second-order temporal
accuracy at the physical $\gamma=10^{-3}$, spectral spatial accuracy on the
smooth test problems considered, agreement with an independent nonlinear
solver, and a clean cost comparison.

\end{enumerate}
The present work is complementary to that of Ref.\ \cite{Brunk2025}, not an
improvement on its analysis; we do not claim a rigorous error estimate.

\paragraph{Outline.} \Cref{sec:cont-energy} derives the continuous energy
law and its SAV reformulation. \Cref{sec:scheme} states the discrete
scheme. \Cref{sec:thm} proves the discrete energy identity, the paper's main
result. \Cref{sec:structure} establishes solvability and mass conservation.
\Cref{sec:accuracy} discusses accuracy and the
Fourier--Galerkin implementation. \Cref{sec:numerics} collects the
numerical evidence: the identity to round-off, convergence, agreement with
independent solvers, cost, and the coupled-vs-uncoupled demonstration.
Proofs are given as concise sketches in the main text; fuller derivations
are collected in \Cref{app:derivations}.

\section{The continuous energy law and its SAV reformulation}\label{sec:cont-energy}

We first establish the energy-decay property of the continuous model \eqref{eq:chq}, introduced in \Cref{sec:energy}. We show not only that the energy $E$ decreases over time, but also derive an explicit decay rate in terms of the model's coefficients. The numerical scheme in \Cref{sec:scheme} will be designed to reproduce this decay rate exactly at the discrete level.

\begin{proposition}\label{prop:cont-energy}
Let $(\phi,\mu,q)$ solve \eqref{eq:chq}. Then $\tfrac{\mathrm d}{\mathrm dt}E=-\Dcont$ with
\begin{equation}\label{eq:cont-D}
\Dcont=\int_\Om\Bigl[\bigl(b(\phi)-\tfrac{c(\phi)^2}{d_0}\bigr)|\grad\mu|^2
   +\tfrac1{d_0}\bigl|c(\phi)\grad\mu-d_0\grad(A(\phi)q)\bigr|^2
   +\kappa(\phi)q^2+\epsq|\grad q|^2\Bigr]\dx ,
\end{equation}
and $\Dcont\ge0$ whenever $b(\phi)\,d_0\ge c(\phi)^2$ pointwise, the
coercivity condition \eqref{eq:coerc-cont}.
\end{proposition}

\begin{proof}
Differentiating $E$, substituting \eqref{eq:chq-mu}, and integrating by parts
gives $\tfrac{\mathrm d}{\mathrm dt}E=\inner{\mu}{\pt\phi}+\inner{q}{\pt q}$.
Pairing \eqref{eq:chq-phi} with $\mu$ and \eqref{eq:chq-q} with $q$ (that is,
multiplying each equation by the stated function and integrating over
$\Om$), the two resulting cross terms are equal rather than
cancelling, by the sign convention adopted in \eqref{eq:chq-phi}. The
elementary identity
\begin{equation}\label{eq:CS}
b|u|^2-2c\,(u\cdot w)+d_0|w|^2
=\Bigl(b-\tfrac{c^2}{d_0}\Bigr)|u|^2+\tfrac1{d_0}\bigl|c\,u-d_0 w\bigr|^2
\end{equation}
applied pointwise with $u=\grad\mu$, $w=\grad(Aq)$ then gives
\eqref{eq:cont-D}. This continuous dissipation identity is already implicit
in \cite{Brunk2025,Spiller2021}. We record it here, with the sign
convention and coercivity threshold made explicit, as the baseline the
discrete scheme of \Cref{sec:scheme} is required to reproduce. Full proof
in \Cref{app:derivations}.
\end{proof}

The energy $E$ satisfies $\tfrac{\mathrm d}{\mathrm dt}E=-\Dcont$  but nothing guarantees $\Dcont\ge0$ for an arbitrary choice of the coefficients $b,c,A,\kappa$. Checking this directly is exactly what produces the one structural condition the coefficients must satisfy. The first two terms in the dissipation rate \eqref{eq:cont-D} form a
quadratic expression in $\grad\mu$ and $\grad(A(\phi)q)$. By
\eqref{eq:CS}, this expression is non-negative for arbitrary values
of these gradients if and only if
\begin{equation}\label{eq:coerc-cont}
b(\phi)d_0\geq c(\phi)^2
\qquad\text{for all admissible }\phi.
\end{equation}
The remaining terms in \eqref{eq:cont-D} are non-negative because
$\kappa(\phi)>0$ and $\epsq>0$. Thus, the coercivity condition
\eqref{eq:coerc-cont} guarantees that the continuous energy is
non-increasing. This condition is Assumption (A1) of
\cite{Brunk2025}.

\subsection{The SAV reformulation}
Proposition~\ref{prop:cont-energy} establishes the continuous energy law. We now focus on discretising \eqref{eq:chq}. The only nonlinear term is $f'(\phi)$ in the $\mu$-equation \eqref{eq:chq-mu}, which comes from the double-well potential $f$. The scalar auxiliary variable (SAV) method removes this nonlinearity by introducing a scalar unknown $r(t)$, a single real number at each time, rather than a new field over $\Om$. The resulting enlarged system has exactly the same solutions as the original system, but no longer contains $f'(\phi)$ explicitly.

The construction starts by separating the part of the energy that is
already at most quadratic from the part that carries the nonlinearity.
Specifically, we write $E=\Elin+\Ef$ with
\[
\Elin\eqdef\int_\Om\Bigl[\tfrac\gamma2|\grad\phi|^2+\tfrac12q^2\Bigr]\dx ,
\qquad
\Ef\eqdef\int_\Om f(\phi)\,\dx .
\]
Since $f$ is bounded below, we may fix a constant $\CO$ large enough that
$\Ef(\phi)+\CO\ge C_1>0$ for every admissible $\phi$ (any $\CO>0$ works for
our benchmark, where $f\ge0$ already; we use $\CO=1$); this positivity is
needed so that the square root 
\[
r(t)\eqdef\sqrt{\Ef(\phi(t))+\CO} ,
\qquad
G(\phi)\eqdef\frac{f'(\phi)}{\sqrt{\Ef(\phi)+\CO}} 
\]
is always well defined. Here $r(t)$ is a single scalar that tracks, at every instant, the size of
the nonlinear part of the energy along the solution, and $G$ is a new
coefficient function chosen for exactly one reason: by construction,
$r\,G(\phi)=f'(\phi)$ whenever $r$ happens to take the value
$\sqrt{\Ef(\phi)+\CO}$. This suggests replacing the troublesome $f'(\phi)$
in \eqref{eq:chq-mu} by the product $r\,G(\phi)$ and treating $r$ as an
independent unknown, governed by its own evolution equation,
\[
\dot r=\half\inner{G(\phi)}{\pt\phi} ,
\]
appended to the system. The evolution equation for $r$ follows directly from differentiating $r=\sqrt{\Ef(\phi)+\CO}$ using the chain rule (\Cref{app:derivations}). Consequently, if this relation holds initially, it is preserved throughout the evolution. Introducing $r$ also allows $G(\phi)$ to be treated as a coefficient function, alongside $b,c,A,$ and $\kappa$. In the discrete scheme of \Cref{sec:scheme}, these coefficients are evaluated at the extrapolated state. This reformulation makes the $\mu$-equation linear in the joint unknowns $(\phi,\mu,q,r)$, regardless of the form of the double-well potential $f$.

This substitution is legitimate because the enlarged system has
exactly the same solutions as the original \eqref{eq:chq}. We show in \Cref{app:derivations} that the
appended equation for $r$ forces $r(t)\equiv\sqrt{\Ef(\phi(t))+\CO}$ along
every solution, given matching initial data, so that
$r(t)G(\phi(t))\equiv f'(\phi(t))$ identically and the SAV
$\mu$-equation reduces back to \eqref{eq:chq-mu}. Along any such solution
the modified energy
\[
\Ered\eqdef\Elin+r^2-\CO
\]
is therefore identically equal to the true energy $E$, and satisfies a
dissipation law of exactly the same completed-square form as before,
$\tfrac{\mathrm d}{\mathrm dt}\Ered=-\Dmod\le0$, with $\Dmod$ built from
\eqref{eq:cont-D} in the same way, only with $f'(\phi)$ replaced by
$rG(\phi)$ throughout. Note that $q$'s
contribution to the energy, $\tfrac12q^2$, is already quadratic, so it
needs no auxiliary variable of its own; only the single genuinely
nonlinear piece $f$ is quadratised, and a single extra scalar $r$ suffices
for the whole system, regardless of how finely $\Om$ is resolved.

One caveat carries forward to the numerical scheme and is worth flagging
now. The identity $r(t)\equiv\sqrt{\Ef(\phi(t))+\CO}$ above holds exactly
only because, in continuous time, two quantities that start out equal and
change at the same rate at every instant must stay equal forever. 
Once time is replaced by discrete steps (\Cref{sec:scheme}), $r^{n+1}$ is no
longer obtained from that continuous argument: it is computed by advancing
a difference equation one step at a time, and nothing forces this
step-by-step update to reproduce $\sqrt{\Ef(\phi^{n+1})+\CO}$ exactly. The
two quantities, $r^n$ and $\sqrt{\Ef(\phi^n)+\CO}$, can therefore differ by
a small, controlled amount. This is exactly where the modified and true
discrete energies can differ (\Cref{sec:modvstrue} discusses this in
detail).

\section{The SAV--BDF2 scheme}\label{sec:scheme}

\Cref{sec:cont-energy} established that the continuous, SAV-reformulated
system dissipates its modified energy exactly, with no restriction on the
coefficients beyond coercivity. Turning this into a usable numerical
method needs two further ingredients: a
way to approximate the time derivatives $\pt\phi,\pt q,\dot r$ to second
order in the step size $\dt$, and a way to evaluate the phase-dependent
coefficients $b,c,A,\kappa,G$ without reintroducing the very nonlinearity
the SAV device was meant to remove. The first is supplied by a standard
second-order finite-difference formula in time, BDF2. The second is this
paper's central device, already previewed in \Cref{sec:positioning}:
evaluate every coefficient not at the unknown new state $\phi^{n+1}$, but
at an already-known approximation to it, extrapolated from the two
previous time levels. This section makes both ingredients precise and
states the resulting scheme. Then in \Cref{sec:thm}, we show that combining them
preserves the energy identity exactly.

\subsection{Notation}\label{sec:notation}
We define the scalar product as $\inner{u}{v}\eqdef\int_\Om uv\dx$ and the norm as $\norm{u}^2\eqdef\inner{u}{u}$. On the periodic box, with no boundary terms, we have

\begin{subequations}\label{eq:ibp}
\begin{align}
\inner{\dv F}{g}&=-\inner{F}{\grad g}, \label{eq:ibp1}\\
\inner{\lap u}{g}&=-\inner{\grad u}{\grad g} . \label{eq:ibp2}
\end{align}
\end{subequations}
The second identity follows from the first by taking $F=\grad u$.

Time derivatives are approximated by BDF2, the second-order backward
differentiation formula: a standard finite-difference rule that
approximates $\dot u(t_{n+1})$ using the current level together with the
two preceding ones, $u^{n+1},u^n,u^{n-1}$ (as opposed to the
simpler, only first-order-accurate backward-Euler rule, which uses just
$u^{n+1},u^n$). Write $t_n=n\dt$ for the time levels and $u^n\approx u(t_n)$
for a numerical approximation at level $n$; a superscript $\str$ will
denote a quantity frozen at the extrapolated state $\phi^\str$, introduced
below. Specifically, we define the BDF2 first and second time differences
\begin{equation}\label{eq:bdf2-ops}
\Dtwo u^{n+1}\eqdef\frac{3u^{n+1}-4u^{n}+u^{n-1}}{2\dt},
\qquad
\dd u^{n+1}\eqdef u^{n+1}-2u^{n}+u^{n-1}.
\end{equation}
$\Dtwo u^{n+1}$ is the actual BDF2 approximation to $\dot u(t_{n+1})$,
accurate to $O(\dt^2)$; $\dd u^{n+1}$ is not itself an approximation to
anything, but a bookkeeping quantity that will appear naturally, in
\Cref{sec:thm}, as the non-negative remainder left over after telescoping
a BDF2 term into a discrete energy difference.

\subsection{Freezing the coefficients}

The one remaining obstacle to a linear scheme is that $b,c,A,\kappa$ (and,
after \Cref{sec:cont-energy}'s reformulation, $G$) are themselves
functions of $\phi$. Evaluating them at the unknown $\phi^{n+1}$ being
solved for would make every equation below nonlinear in $\phi^{n+1}$,
exactly the nonlinearity a Newton-type solver would have to iterate on
instead; evaluating them at an already known approximation to
$\phi(t_{n+1})$ keeps every equation linear in the actual unknowns, at the
cost of a small, controllable extrapolation error. Since $\phi^n,\phi^{n-1}$
are already known once we reach time level $n+1$, we define the extrapolation
\[
\phi^\str\eqdef2\phi^n-\phi^{n-1} ,
\]
a new value recomputed at every step from the two preceding ones. So, $\phi^\str$ depends on $n$ exactly as $\phi^n$ does, but we write $\phi^\str$
rather than $\phi^{\star,n+1}$ purely to keep the notation light. $\phi^\str$  approximates $\phi(t_{n+1})$ to $O(\dt^2)$, of the same order as that achieved by BDF2.
However, using $\phi^\str=\phi^n$ instead would only be accurate to
$O(\dt)$ and would degrade the scheme's overall accuracy. For the
specific coefficient forms of the benchmark problem (\cite{Brunk2025},
given in full in \Cref{sec:numerics}), $A(\phi)\sim\cot(\pi\phi)$ and
$\kappa(\phi)$ blow up as $\phi\to0,1$; to guard against $\phi^\str$
straying outside the physically admissible range before the scheme has had
a chance to correct it, we additionally clamp it to the interval
$[\delta_c,1-\delta_c]$,
\[
\phi^\str_b\eqdef\clamp(\phi^\str,\delta_c,1-\delta_c)
\eqdef\min\bigl(\max(\phi^\str,\delta_c),\,1-\delta_c\bigr) ,
\qquad \delta_c\sim10^{-3}
\]
(a numerical safeguard only, verified inactive throughout the convergence
regime studied below). Every phase-dependent coefficient is then frozen at
this same argument,
\begin{equation}\label{eq:frozen}
\bs\eqdef b(\phi^\str_b),\ \
\cs\eqdef c(\phi^\str_b),\ \
\As\eqdef A(\phi^\str_b),\ \
\ks\eqdef\kappa(\phi^\str_b),\ \
\Gs\eqdef\frac{f'(\phi^\str)}{\sqrt{\Ef(\phi^\str)+\CO}} ,
\end{equation}
using the same clamped state for all four of $b,c,A,\kappa$, for
uniformity, although strictly the clamp only matters for $A,\kappa$, whose
specific forms would otherwise be singular. The parameters $b,c$ remain well behaved throughout the range for the coefficient forms used here and $\Gs$ uses the
unclamped $\phi^\str$ throughout, since $f'$ and $\Ef$ do not develop
singularities either. This is the step, previewed in
\Cref{sec:positioning}, on which the rest of the paper turns: freezing
all of $b,c,A,\kappa$ at the same state, rather than at
several different ones or at $\phi^{n+1}$ itself, is exactly what
preserves the coupled cross-diffusion structure of \eqref{eq:chq} intact,
as \Cref{sec:thm} shows next.

\begin{assumption}[Frozen coercivity]\label{ass:coerc}
We assume that $\bs(x)-\cs(x)^2/d_0\ge0$ for all $x$. This is the same coercivity condition as \eqref{eq:coerc-cont}, evaluated at the frozen, clamped state $s=\phi^\str_b(x)$. For the benchmark coefficient set, the condition is satisfied with $\bs-\cs^2/d_0=\epsb$ at every point. Thus, the coercivity margin is strictly positive and independent of the time step $\dt$.
\end{assumption}

\subsection{The discrete scheme}
The scheme advances only the variables $\phi$, $q$, and $r$, since these are the three quantities governed by evolution equations in \eqref{eq:chq}. Accordingly, the scheme must retain their values at the previous two time levels. The chemical potential $\mu$, by contrast, is defined algebraically by \eqref{eq:chq-mu} as an instantaneous function of $\phi$ and, following the SAV reformulation, $r$, at the current time level. Since $\mu$ has no independent evolution equation, its value is recomputed at each time step from the current values of these variables. Consequently, neither $\mu^n$ nor $\mu^{n-1}$ is required by the scheme, and neither appears in the formulation below.

Collecting the pieces above, the scheme is obtained from the continuous
system \eqref{eq:chq} by three simultaneous replacements: every time
derivative $\partial_t$ is replaced by the BDF2 operator $\Dtwo$, every
phase-dependent coefficient is replaced by its frozen value
\eqref{eq:frozen}, and the SAV auxiliary equation for $r$ from
\Cref{sec:cont-energy} is adjoined as a fourth equation. We label the
resulting four equations (S1)--(S4) (``S'' for ``scheme'') to keep them
visually distinct from the continuous model's \eqref{eq:chq-phi}--\eqref{eq:chq-q}.
For $n\ge1$, given $(\phi^{n-1},q^{n-1},r^{n-1})$ and $(\phi^{n},q^{n},r^{n})$,
we find $(\phi^{n+1},\mu^{n+1},q^{n+1},r^{n+1})$ solving
\begin{subequations}\label{eq:S}
\begin{align}
\Dtwo\phi^{n+1} &= \dv\!\bigl(\bs\grad\mu^{n+1}-\cs\grad(\As q^{n+1})\bigr),
   \tag{S1}\label{eq:S1}\\
\mu^{n+1} &= -\gamma\lap\phi^{n+1}+r^{n+1}\Gs, \tag{S2}\label{eq:S2}\\
\Dtwo q^{n+1} &= -\ks q^{n+1}
   -\As\,\dv\!\bigl(\cs\grad\mu^{n+1}-d_0\grad(\As q^{n+1})\bigr)+\epsq\lap q^{n+1},
   \tag{S3}\label{eq:S3}\\
\Dtwo r^{n+1} &= \half\inner{\Gs}{\Dtwo\phi^{n+1}}. \tag{S4}\label{eq:S4}
\end{align}
\end{subequations}
By construction, each equation is linear in the unknowns
$(\phi^{n+1},\mu^{n+1},q^{n+1},r^{n+1})$, since all coefficients are
evaluated using the frozen, already-known quantities in \eqref{eq:frozen},
rather than the unknown $\phi^{n+1}$ itself. The only coupling that
prevents the four equations from being block-diagonal arises in
\eqref{eq:S4}, where the scalar unknown $r^{n+1}$ is coupled to all
degrees of freedom of $\phi^{n+1}$ through the single scalar quantity
$\inner{\Gs}{\Dtwo\phi^{n+1}}$. \Cref{sec:structure} exploits this
specific structure to avoid solving the full coupled system simultaneously.

BDF2 itself needs two prior levels to get started, so the very first step,
from $(\phi^0,q^0,r^0)$ to $(\phi^1,q^1,r^1)$, cannot use \eqS
directly. We use a first-order backward-Euler (BE) self-starting step by
default: a simpler, one-previous-level scheme whose local error is
$O(\dt^2)$, small enough that it does not degrade the second-order global
rate of the rest of the trajectory (checked numerically,
\Cref{sec:temporal}). A Richardson-extrapolated variant, built from two
half-steps of BE, is also available for a formally second-order-accurate
start.

\section{The discrete energy identity}\label{sec:thm}

Discretising a dissipative system can destroy the property that makes it dissipative. A time-stepping scheme does not automatically inherit the energy decay of the continuous system it approximates. This is a real concern for our scheme: the freezing of the coefficients (\Cref{sec:scheme}) makes the equations linear, but it could have cost us the energy law. The discrete energy might then decrease only approximately, with an error that grows with $\dt$, or it might decrease only for sufficiently small $\dt$.

The scheme \eqS satisfies an exact discrete analogue of \Cref{prop:cont-energy}, with no approximation and no restriction on $\dt$. Specifically, we construct a discrete energy $\Emod^n$ and a discrete dissipation $\Dmod^{n+1}$, which have the same completed-square structure as the continuous dissipation $\Dcont$. They satisfy

$$\Emod^{n+1}-\Emod^n = -\dt\,\Dmod^{n+1} - R^{n+1},$$

where $R^{n+1}\ge0$ is an explicit non-negative remainder. This holds exactly, for every $\dt>0$. We call a scheme unconditionally energy stable in this sense: the modified energy decreases at every step, by construction, for any step size, and not only once $\dt$ is sufficiently small.

Since BDF2 is a two-step method, the discrete energy $\Emod^n$ depends on two consecutive time levels, through the quadratic form $\Phimix$ (\Cref{sec:notation}). The result relies on the coefficient freezing of \Cref{sec:scheme}. Because all coefficients in \eqS are evaluated at the same extrapolated state, the discrete calculation parallels the continuous proof of \Cref{prop:cont-energy} step by step, and no additional error terms need to be controlled.

\begin{theorem}[Discrete energy--dissipation identity]\label{thm:main}
Assume $d_0>0$ constant, $\epsq\ge0$, $\ks\ge0$ pointwise (i.e.\ at every
point $x\in\Om$, not merely on average), \Cref{ass:coerc}, and a spatial
discretisation with exact discrete integration by parts under which
$\grad\Dtwo=\Dtwo\grad$ (differentiating in space and then applying the
BDF2 time-difference operator gives the same result as doing these two
steps in the opposite order) and the frozen-coefficient products entering
\eqref{eq:CS} are evaluated and integrated exactly at each degree of
freedom, that is, computing a product such as $\bs|\grad\mu^{n+1}|^2$
pointwise at the grid points (or retained Fourier modes) and then summing
it discretely gives exactly the same number the continuous integral would,
with no quadrature or aliasing error. This holds for the periodic
Fourier--Galerkin discretisation without dealiasing used below. With $\Phimix(u,v)\eqdef\norm{u}^2+\norm{2u-v}^2$,
\begin{align}
\Emod^n &\eqdef \tfrac14\Bigl[\gamma\,\Phimix(\grad\phi^n,\grad\phi^{n-1})
   +\Phimix(q^n,q^{n-1})+2\,\Phimix(r^n,r^{n-1})\Bigr]-\CO ,
   \label{eq:Emod}\\
\Dmod^{n+1} &\eqdef \int_\Om\Bigl(\bs-\tfrac{\cs^2}{d_0}\Bigr)|\grad\mu^{n+1}|^2
   +\tfrac1{d_0}\int_\Om\bigl|\cs\grad\mu^{n+1}-d_0\grad(\As q^{n+1})\bigr|^2\notag\\
   &\qquad{}+\inner{\ks q^{n+1}}{q^{n+1}}+\epsq\norm{\grad q^{n+1}}^2 \ \ (\ge0),
   \label{eq:Dmod}
\end{align}
solutions of \eqS satisfy, for \emph{every} $\dt>0$,
\begin{equation}\label{eq:star}
\Emod^{n+1}-\Emod^n
+\tfrac14\Bigl[\gamma\norm{\grad\dd\phi^{n+1}}^2+\norm{\dd q^{n+1}}^2
   +2\,(\dd r^{n+1})^2\Bigr]
+\dt\,\Dmod^{n+1}
=0 .
\end{equation}
\end{theorem}

\begin{proof}
Testing \eqref{eq:S2} with $\Dtwo\phi^{n+1}$ and \eqref{eq:S1} with
$\mu^{n+1}$ gives two expressions for the same pairing
$\inner{\mu^{n+1}}{\Dtwo\phi^{n+1}}$; equating them, and using \eqref{eq:S4}
to replace $r^{n+1}\inner{\Gs}{\Dtwo\phi^{n+1}}$ by $2r^{n+1}\Dtwo r^{n+1}$,
isolates the frozen-coefficient terms on one side. Adding the result of
testing \eqref{eq:S3} with $q^{n+1}$, the two cross terms combine (add, as in
\Cref{prop:cont-energy}'s proof, because $\bs,\cs,\As,\ks$ are all evaluated
at the same $\phi^\str_b$), and \eqref{eq:CS} converts them into
exactly $-\Dmod^{n+1}$, with no $\|\phi^{n+1}-\phi^\str\|$ remainder, since
the frozen coefficients carry no $\phi^{n+1}$-dependence. The G-stability
concept is due to Dahlquist \cite{Dahlquist1976}; the specific multiplier
identity used here, for BDF2, is due to Nevanlinna and Odeh
\cite{NevanlinnaOdeh1981}:
\begin{equation}\label{eq:G}
\inner{3a^{n+1}-4a^n+a^{n-1}}{a^{n+1}}
=\half\bigl[\Phimix(a^{n+1},a^n)-\Phimix(a^n,a^{n-1})\bigr]
 +\half\norm{a^{n+1}-2a^n+a^{n-1}}^2 ,
\end{equation}
applied to $a=\grad\phi,q,r$, then turns the left-hand side into the
telescoping difference $\Emod^{n+1}-\Emod^n$ plus the stated non-negative
remainder, giving \eqref{eq:star}. Non-negativity of $\Dmod^{n+1}$ follows
from \Cref{ass:coerc}, $\ks\ge0$, $\epsq\ge0$. Full proof in
\Cref{app:derivations}.
\end{proof}

\begin{corollary}[Unconditional energy stability]\label{cor:stab}
Since every term after $\Emod^{n+1}-\Emod^n$ in \eqref{eq:star} is
non-negative,
\[
\Emod^{n+1}\le\Emod^n \qquad\text{for every }\dt>0 .
\]
That is, the SAV--BDF2 scheme is unconditionally energy stable: the
modified discrete energy never increases, at any step size. Since
additionally $\Emod^n\ge-\CO$ is bounded below, the monotone sequence
$\{\Emod^n\}$ converges.
\end{corollary}

\begin{remark}[The discrete energy bounds the solution]\label{rem:bounded}
Corollary~\ref{cor:stab} says that the modified energy $\mathcal{E}^n$ never increases. A decreasing energy is only useful if it also limits how large the fields can become. This is indeed the case.

The energy $\mathcal{E}^n$ is a sum of squares of $\nabla\phi$, $q$ and $r$ at the two stored time levels $n$ and $n-1$. The two levels are combined through the form $\Phi(u,v)=\|u\|^2+\|2u-v\|^2$. This combination is positive: it vanishes only if $u=v=0$, and it is bounded below and above by multiples of $\|u\|^2+\|v\|^2$. Consequently, $\mathcal{E}^n+C_0$ is never smaller than a fixed positive multiple of
\[
  \gamma\left(\|\nabla\phi^n\|^2+\|\nabla\phi^{n-1}\|^2\right)
  + \|q^n\|^2+\|q^{n-1}\|^2
  + 2\left((r^n)^2+(r^{n-1})^2\right).
\]
Since $\mathcal{E}^n\le\mathcal{E}^1$ for all $n\ge1$, none of these terms can exceed a fixed multiple of $\mathcal{E}^1+C_0$. Hence $\|\nabla\phi^n\|$, $\|q^n\|$ and $|r^n|$ stay bounded for all times and for every $\Delta t>0$.

Of the three quantities just bounded, $r$ already carries the effect of
the bulk term $\Ef(\phi)$ by construction. What remains uncontrolled is
$\phi$ itself: $\Emod^n$ reaches $\phi$ only through $\grad\phi$, and a
gradient cannot see an additive constant, since $\phi$ and
$\phi+\text{const}$ have the same gradient. The size of $\phi$ itself is
therefore controlled separately.
By mass conservation (Proposition~\ref{prop:mass}), the spatial average of $\phi^n$ equals its initial value at every step. A field with fixed average and bounded
gradient on a periodic domain cannot be arbitrarily large (Poincar\'e
inequality), so $\|\phi^n\|$ is bounded as well.

In summary, the numerical solution cannot blow up: $\phi$, $q$ and $r$ remain bounded at all times, whatever the time step.
\end{remark}

\begin{remark}[Stability of the filtered scheme actually used]\label{rmk:filter}
The long-time physics runs (\Cref{sec:physics}) apply an optional $2/3$-rule
spectral filter $F$ to $(\phi,q)$ after every step (removing aliasing noise
from evaluating nonlinear coefficients on the collocation grid). Since $F$
is an orthogonal, $\grad$-commuting projection, a short argument (filtering
cannot increase $\Phimix$: $\Phimix(Fu,v)=\Phimix(u,v)-5\norm{(I-F)u}^2$ for
$Fv=v$; \Cref{app:derivations}) shows \eqref{eq:star} extends exactly
to the filtered scheme, with two explicit additional non-negative terms
penalising the filtered-out content. Hence, the scheme actually producing the
physics results shown in \Cref{sec:physics} is itself provably (not just empirically) stable for the same energy. We also confirmed numerically that both 
terms stay strictly positive, with genuine margin and not merely at
round-off, over real trajectories.
\end{remark}

\begin{remark}[Stability $\ne$ accuracy]
Equation \eqref{eq:star} does not show that $\Dmod^{n+1}$, built from the extrapolated
mobilities, approximates the physical dissipation $\Dcont$ to second order
under all circumstances. That is a consistency question settled,
numerically in \Cref{sec:numerics}. Also, $\Emod^n\ne E(\phi^n,q^n)$: the
$\Phimix$-form and $(r^n)^2-\CO$ each differ from the corresponding true
quantities by $O(\dt^2)$ (measured $\approx7\times10^{-6}<\dt^2$), so
$\Emod^n$ decreases monotonically while $E(\phi^n,q^n)$ may show $O(\dt^2)$
non-monotonicities. Every figure claiming ``energy decay'' states which
energy.\label{sec:modvstrue}
\end{remark}

\section{Structural properties}\label{sec:structure}

\Cref{sec:thm} showed that \emph{if} a solution of \eqS exists at
a given step, the discrete energy identity forces it to be well behaved:
the modified energy can only decrease. That conditional phrasing hides a
question energy stability alone does not answer: does \eqS, a
linear system coupling four unknown fields at every step, actually
\emph{have} a solution there, and is it unique? A scheme that behaves well
whenever it has a solution is of limited practical use if the underlying
linear system can fail to be solvable, or have more than one solution, at
some step sizes or some states; an implementation would then have no
guarantee that the linear solve it calls every step is even well posed.
This section closes that gap. \Cref{thm:solv} shows the linear system is
uniquely solvable at \emph{every} step, for \emph{every} $\dt>0$, reusing
the energy identity of \Cref{thm:main} itself as the proof's main tool.
\Cref{prop:mass} records a second structural property, exact mass
conservation, needed along the way.

\begin{theorem}[Unconditional unique solvability]\label{thm:solv}
Under the same assumptions as \Cref{thm:main}, and additionally
$\bs\ge b_1>0$, the linear system \eqS is uniquely solvable for every
$\dt>0$.
\end{theorem}

\begin{proof}
Since \eqS is linear, finite-dimensional, and square, existence
and uniqueness for every choice of data are equivalent to the
homogeneous system (all data set to zero) having only the trivial
solution. So take a solution of \eqS with
$\phi^n=\phi^{n-1}=q^n=q^{n-1}=r^n=r^{n-1}=0$. Plugging this data into
\eqref{eq:star} collapses it to a sum of non-negative terms equal to
zero, which forces $\grad\phi^{n+1}=q^{n+1}=r^{n+1}=0$ and
$\Dmod^{n+1}=0$. Mass conservation (\Cref{prop:mass}) then pins
$\phi^{n+1}\equiv0$. With $q^{n+1}=0$, the condition $\Dmod^{n+1}=0$
reduces to $\int_\Om\bs|\grad\mu^{n+1}|^2=0$. Since $\bs\ge b_1>0$, this
forces $\grad\mu^{n+1}=0$, so $\mu^{n+1}$ is spatially constant, though
not yet known to be zero. Equation \eqref{eq:S2}, with $\phi^{n+1}\equiv0$
and $r^{n+1}=0$ already established, gives $\mu^{n+1}=0$ directly,
pinning that constant. So the homogeneous system has only the trivial
solution, and \eqS is uniquely solvable for every $\dt>0$. Full
proof in \Cref{app:derivations}.
\end{proof}

\begin{proposition}[Exact mass conservation]\label{prop:mass}
$M^n\eqdef\int_\Om\phi^n$ satisfies $M^n\equiv M^0$ for every $n$.
\end{proposition}

\begin{proof}
Test \eqref{eq:S1} with the constant function $1$. By \eqref{eq:ibp1}
with $g=1$, $\inner{\dv F}{1}=-\inner{F}{\grad 1}=0$ for any flux
$F$, since the gradient of a constant is zero; so the right-hand side of
\eqref{eq:S1} integrates to exactly zero, whatever $\mu^{n+1},q^{n+1}$
turn out to be. The left-hand side, integrated, is
$\int_\Om\Dtwo\phi^{n+1}=\tfrac1{2\dt}\bigl(3M^{n+1}-4M^n+M^{n-1}\bigr)$
by the definition \eqref{eq:bdf2-ops} of $\Dtwo$, so
\begin{equation}\label{eq:mass-rec}
3M^{n+1}-4M^n+M^{n-1}=0
\end{equation}
at every step $n\ge1$: a linear recursion relating three consecutive mass
values, forced purely by the structure of \eqref{eq:S1}, regardless of
$\mu^{n+1},q^{n+1}$. The same argument, applied to the (simpler)
backward-Euler startup step, gives $M^1=M^0$. We now show
$M^n\equiv M^0$ for every $n\ge0$ by induction: this holds for $n=0,1$ by
definition and by the startup property just established. Suppose it
holds up to some $n\ge1$, i.e.\ $M^{n-1}=M^n=M^0$; then
\eqref{eq:mass-rec} gives $3M^{n+1}=4M^n-M^{n-1}=4M^0-M^0=3M^0$, i.e.\
$M^{n+1}=M^0$, extending the claim to $n+1$. Hence $M^n\equiv M^0$ for
every $n$.
\end{proof}

\section{Accuracy and implementation}\label{sec:accuracy}

Theorems~\ref{thm:main} and \ref{thm:solv} establish that the scheme is
unconditionally stable and solvable; neither says anything about whether
it is actually \emph{accurate}, or what turning
\eqref{eq:S1}--\eqref{eq:S4} into working code involves. This section
addresses both in turn. \Cref{alg:step} summarises the computation
performed at every time step of the implementation used throughout the
rest of the paper; the text below justifies why that computation is
second-order accurate, and fills in the specific spatial discretisation
and linear solver \Cref{alg:step} leaves unspecified.

\begin{algobox}[t]
\centering
\fbox{\begin{minipage}{0.93\linewidth}
\textbf{One time step of the SAV--BDF2 scheme, level $n\to n+1$ (for $n\ge1$).}
\begin{enumerate}
\item \textbf{Extrapolate:} $\phi^\star \leftarrow 2\phi^n-\phi^{n-1}$.
\item \textbf{Clamp:} $\phi^\star_b \leftarrow \clamp(\phi^\star,\delta_c,1-\delta_c)$.
\item \textbf{Freeze coefficients:}\\
$\bs,\cs,\As,\ks \leftarrow b(\phi^\star_b),\,c(\phi^\star_b),\,A(\phi^\star_b),\,\kappa(\phi^\star_b)$,
\quad $\Gs \leftarrow f'(\phi^\star)/\sqrt{\Ef(\phi^\star)+\CO}$.
\item \textbf{Solve} the linear system \eqref{eq:S1}--\eqref{eq:S4} for
$(\phi^{n+1},\mu^{n+1},q^{n+1},r^{n+1})$.
\item \textbf{Filter (optional):} $\phi^{n+1}\leftarrow F\phi^{n+1}$,
$q^{n+1}\leftarrow Fq^{n+1}$ ($2/3$-rule; long-time runs of
\Cref{sec:physics} only, \Cref{rmk:filter}).
\item Advance $n\leftarrow n+1$ and repeat until $t=T$.
\end{enumerate}
\end{minipage}}
\caption{The computation performed at every step. Steps 1--3 cost
$O(N^2\log N)$ (a handful of FFTs); step 4 dominates the per-step cost
(\Cref{sec:cost}), by one of the two methods described below. The very
first step ($n=0$) skips steps 1--3 ($\phi^\star:=\phi^0$, nothing to
extrapolate yet) and replaces step 4 by the first-order backward-Euler
start of \Cref{sec:scheme}, or its Richardson-extrapolated variant.}
\label{alg:step}
\end{algobox}

\paragraph{Accuracy.} The scheme's formal order follows from a standard
truncation argument, combining three separate sources of $O(\dt^2)$ error
into one. First, BDF2 itself contributes a local truncation error of
$O(\dt^2)$, by construction. Second, the extrapolation
$\phi^\star=\phi(t_{n+1})+O(\dt^2)$ (\Cref{sec:scheme}) makes every frozen
coefficient consistent to the same order, \emph{provided} $\phi^\star$
stays inside the clamp range; this proviso is verified to hold throughout
the convergence regime studied below, so the clamp safeguard never
actually bites and never degrades the order. Third, the SAV substitution
itself costs no further accuracy: $r^{n+1}\Gs=f'(\phi^{n+1})+O(\dt^2)$, by
the standard SAV consistency argument \cite{ShenXu2018}. A local
truncation error of $O(\dt^3)$, accumulated over the $O(1/\dt)$ steps of a
fixed time horizon, gives a global error of $O(\dt^2)$ overall; this is
confirmed numerically in \Cref{sec:temporal}, including when the
trajectory is started by the (only first-order-accurate) backward-Euler
step of \Cref{alg:step}. A rigorous relative-energy error estimate, in the
style of \cite{Brunk2025}, is left for future work.

\paragraph{Spatial discretisation.} Space is discretised by periodic
Fourier--Galerkin collocation on an $N\times N$ grid, without $2/3$-rule
dealiasing by default (\Cref{rmk:filter} discusses the optional filter
used for the long-time runs). This choice is exactly what
\Cref{thm:main} needs: exact discrete integration by parts and the
commutation $\grad\Dtwo=\Dtwo\grad$ both hold automatically for a
periodic Fourier discretisation, with no further assumption, since every
pointwise product entering the proof is formed directly on the
collocation grid. One technical point is worth flagging explicitly,
since it is easy to get wrong and silently breaks the identity if missed:
the discrete Laplacian must use the symbol
$-(k_x^2\mathbf1_{k_x\ne N/2}+k_y^2\mathbf1_{k_y\ne N/2})$, zeroing out
the Nyquist-frequency contribution in each direction, rather than the
naive symbol $-|k|^2$. With the naive symbol, the discrete gradient and
Laplacian become subtly inconsistent at the Nyquist mode, and
\eqref{eq:star} fails outright once the solution develops any content at
that frequency.

\paragraph{Solving the linear system.} Step 4 of \Cref{alg:step}, a
$(2N^2+1)$-dimensional linear system, is solved in one of two ways: by a
dense direct solve, exact up to machine precision and so used as a
ground-truth reference on small grids ($N\lesssim40$); or by matrix-free
GMRES (the generalised minimal residual method, an iterative solver for
linear systems that does not require storing the system matrix) with a
Fourier-diagonal preconditioner, built from the frozen coefficients
reduced to
representative scalars, used on the larger grids needed for the physics
study. The two methods agree to $10^{-12}$ wherever both are affordable.
Throughout \Cref{sec:numerics}, an independent nonlinear reference solver
is also used as a cross-check: it treats $f'(\phi^{n+1})$ fully implicitly
by Newton iteration, with no SAV device at all, on the identical spatial
discretisation, so that any disagreement with the SAV scheme isolates
exactly the effect of the SAV reformulation, and nothing else.

\section{Numerical results}\label{sec:numerics}

We verify the scheme using the configuration in \Cref{tab:coef} below
(domain $\Om=(0,1)^2$, interface parameter $\gamma=10^{-3}$), which we call
the \emph{reference configuration} and refer to by that name for the rest
of the paper. All stability, solvability and convergence claims
established above hold regardless of this choice; the reference
configuration is used purely for demonstration. 

\begin{table}[t]
\centering
\caption{The reference configuration ($\Om=(0,1)^2$, $d_0=1$, $\CO=1$
throughout), which satisfies the model's stated structural conditions
(\Cref{eq:coerc-cont}).}
\label{tab:coef}
\begin{tabular}{@{}ll@{}}
\toprule
$c(\phi)$ & $\sqrt{1.6\phi^2(1-\phi)^2+0.1}$ \\
$b(\phi)$ & $(1+\eps_1)c(\phi)^2$, $\eps_1=10^{-3}$ \\
$f(\phi)$ & $3\phi^2(1-\phi)^2$ \\
$\kappa(\phi)$ & $10^{-3}(10\phi^2+10^{-3})^{-1}$ \\
$A(\phi)$ & $\tfrac12\bigl(1+\tanh(10[\cot\pi\phi^*-\cot\pi\phi])\bigr)$, $\phi^*=0.4$ \\
$\gamma=\epsq$ & $10^{-3}$ \\
IC & $\phi_0=0.4+\xi$, $\xi\sim\text{Unif}[-0.0025,0.0025]$, $q_0=0$ \\
\bottomrule
\end{tabular}
\end{table}

\subsection{Headline demonstration: energy, mass, \texorpdfstring{$\Delta t$-robustness}{Delta t-robustness}}\label{sec:headline}

In the style of \cite{LiYang2023}, we run the coupled model from the IC
above, $N=128$, to $T=10$, at three widely separated step sizes spanning two
orders of magnitude, $\dt\in\{10^{-2},10^{-3},10^{-4}\}$ ($10^3,10^4,10^5$
steps respectively), with GMRES solved to a relative tolerance of
$10^{-10}$ at every step.

\begin{figure}[t]\centering
\includegraphics[width=0.7\textwidth]{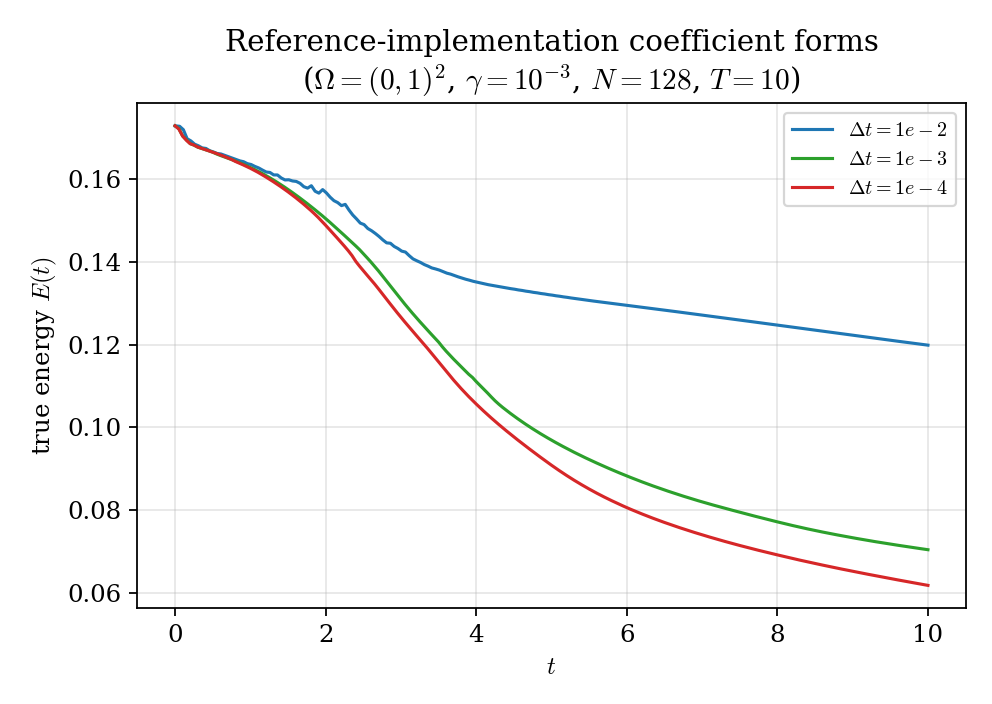}
\caption{True energy $E(t)$, all three $\dt$. The two finer $\dt$ track each
other closely to $t\approx3$ before separating; $\dt=10^{-2}$ diverges from
both much earlier and settles at a visibly higher
energy.}\label{fig:wp9-energy}
\end{figure}

\begin{figure}[t]\centering
\includegraphics[width=0.7\textwidth]{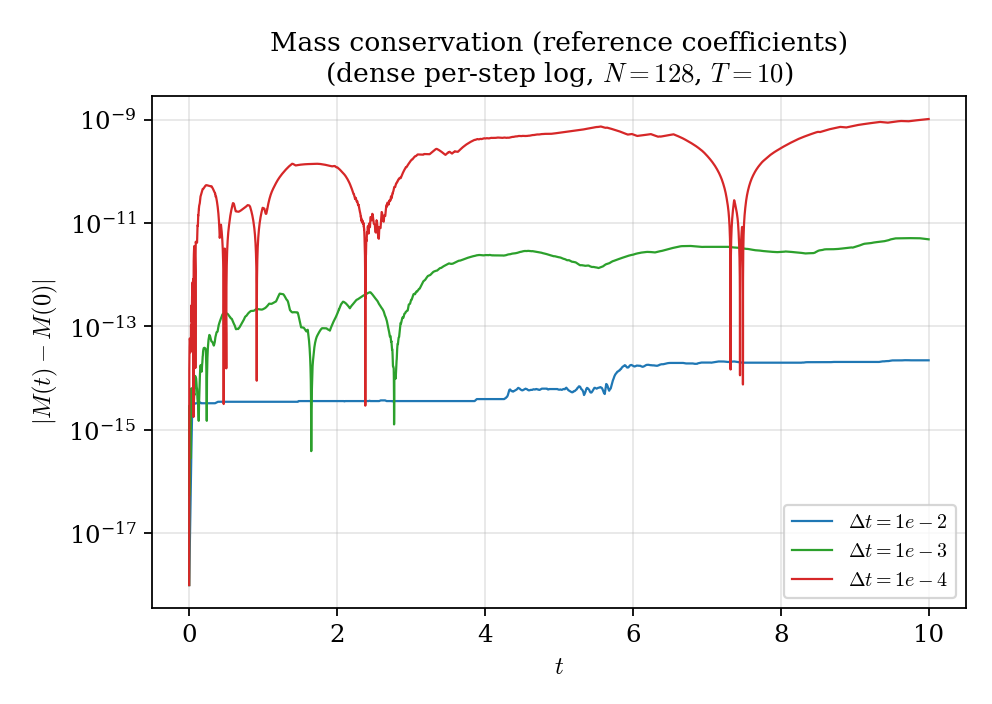}
\caption{Mass conservation $|M(t)-M(0)|$ (dense per-step log, log scale),
same three runs. All stay at or below $10^{-9}$, growing with the number of
steps taken. \Cref{prop:mass} is an identity for the linear system
\eqS \emph{as solved exactly}. Each GMRES solve leaves a residual up to the
$10^{-10}$ tolerance above, so this is the finite-precision floor on an
otherwise exact conservation law, not machine round-off, and it accumulates
mildly over the $10^3$--$10^5$ steps taken rather than staying fixed per
step.}\label{fig:wp9-mass}
\end{figure}

\begin{center}
\footnotesize
\begin{tabular}{@{}rccccc@{}}
\toprule
$\dt$ & steps & $E(10)$ & mass drift (max) & GMRES/step (med., max) & clamp frac (max) \\
\midrule
$10^{-2}$ & $10^3$  & $0.1199$ & $2.2\times10^{-14}$ & $675,\ 845$ & $0.026$ \\
$10^{-3}$ & $10^4$  & $0.0704$ & $5.1\times10^{-12}$ & $333,\ 418$ & $0.034$ \\
$10^{-4}$ & $10^5$  & $0.0618$ & $1.0\times10^{-9}$  & $84,\ 102$  & $0.012$ \\
\bottomrule
\end{tabular}
\end{center}

$\Emod^n$ decreases strictly at every step and mass is conserved to
to within the GMRES tolerance at all three $\dt$, including $\dt=10^{-2}$
(\Cref{fig:wp9-mass}), exactly as \Cref{thm:main} guarantees. But
\Cref{fig:wp9-energy} shows stability does not imply accuracy: the two finer $\dt$
agree closely up to $t\approx3$ before separating by $t=10$ ($0.0704$ vs.\
$0.0618$, a $12\%$ gap), while $\dt=10^{-2}$ diverges from both much earlier
and settles at $E(10)=0.1199$, roughly $2\times$ the finer runs' value,
consistent with its markedly higher GMRES cost (median $675$ iterations/step
vs.\ $84$ at $\dt=10^{-4}$) and clamp activity: the simple Fourier
preconditioner is a poor approximation to the true operator once $A(\phi)$'s
transition layer becomes comparably sharp to $1/N$, so this coarse a $\dt$
genuinely under-resolves the coupled dynamics, even though the discrete
energy identity itself holds throughout. Whether the residual gap between
the two finer $\dt$ is still-shrinking temporal error or reflects the
coarsening dynamics' own sensitive dependence on small perturbations once
neighbouring domains begin to interact is not resolved here.

\subsection{The energy identity, convergence, and independent solvers}

Every algebraic identity used in the proofs above, including the
sign-sensitive completing-the-square step, was additionally checked by
independent computer-algebra verification, as a safeguard against
hand-algebra slips.

We also checked numerically that the linear system \eqS is not merely
solvable in principle (\Cref{thm:solv}) but well conditioned in practice. 
To quantify this, we assembled the matrix corresponding to the system of equations \eqS 
explicitly (of size $2N^2+1$) on a small grid, $N=13$, and computed its singular values. The
matrix depends on the solution only through the frozen coefficients, which are
functions of $\phi^\star$. To avoid depending on the states visited by a
particular simulation, we drew $\phi^\star$ at random, for three independent
draws, and used step sizes $\Delta t\in\{10^{-4},10^{-2},1,10^{2}\}$. 
Across all of these cases, the smallest singular value stays at $\sigma_{\min}\gtrsim4\times10^{-4}$, 
and the condition number stays below $2\times10^6$. 
The matrix therefore is far from singular, and round-off errors are amplified at most by this factor. 

In the dense-solve runs shown in \Cref{fig:star}, the discrete energy identity \eqref{eq:star} holds to
round-off ($\sim10^-14$--$10^-17$) at every step, regardless of whether the initial data is smoothly resolved or deliberately under-resolved and noisy.
Since no GMRES tolerance floor is present, mass likewise stays at round-off.
In the under-resolved case, the clamp safeguard of \Cref{sec:scheme} is 
active at no more than $9\%$ of the spatial grid points. The coarser,
 GMRES-tolerance-limited mass conservation reported in \Cref{sec:headline} is
specific to the larger, iteratively-solved production runs, not a 
general limitation of the scheme.

\begin{figure}[t]\centering
\includegraphics[width=0.55\textwidth]{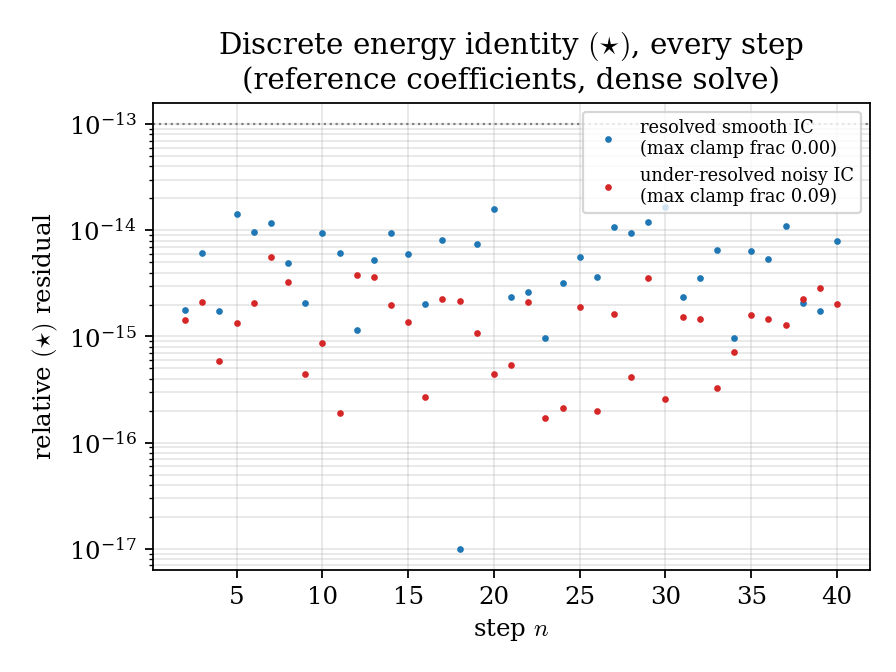}
\caption{Relative $(\star)$ residual \eqref{eq:star}, every step, dense
solve: resolved smooth IC and under-resolved noisy IC (clamp active),
reference coefficients. Both scatter at machine round-off with no
drift.}\label{fig:star}
\end{figure}

\paragraph{Temporal and spatial convergence.}\label{sec:temporal} We check
the scheme's temporal order $p_t$ (the exponent such that the error shrinks
like $\dt^{p_t}$) in two independent ways. The first is the method of
manufactured solutions (MMS) \cite{Roache2002,SalariKnupp2000}, the
standard technique for measuring a scheme's convergence order when no
closed-form solution of the original equations is available: we pick an
arbitrary smooth function of $(x,t)$,
substitute it into \eqref{eq:chq}, and add whatever extra forcing term is
needed to make that function an exact solution of the resulting, slightly
modified system (the explicit function and forcing term used here are
given in \Cref{app:mms}). Because the exact solution is then known in
closed form, we can measure the true numerical error directly at each
$\dt$ and check how fast it shrinks as $\dt\to0$. Doing this with the
reference configuration's own $b,c,A,\kappa$ forms, but with $\gamma=0.30$
in place of $10^{-3}$ and $N=40$ (both chosen large enough that the
spatial grid is not the bottleneck at any $\dt$ tested), gives
$p_t\approx2.00$ for the backward-Euler startup and, for the Richardson
startup, an order climbing from $1.68$ to
$1.95$ as $\dt\to10^{-4}$, approaching $2$ (\Cref{fig:conv}). The second
check is directly at the physical $\gamma=10^{-3}$ ($N=160$, tiny
perturbation IC), where no manufactured solution is used; instead we
estimate the order by
self-convergence, comparing the numerical solutions obtained at
successively halved $\dt$ against one another rather than against a known
exact solution, which gives $p_t=1.98,2.00$. For the spatial
discretisation, the Fourier--Galerkin method is expected to converge
spectrally on smooth data, meaning the error should fall off faster than
any fixed power of $1/N$ as $N$ increases: starting from a smooth but
sharply peaked (Lorentzian-shaped) initial condition and measuring the
error relative to a highly resolved reference run at $N=160$, the error
drops by a factor of $10^7$ as $N$ increases from $16$ to $96$, confirming
this spectral rate.

\begin{figure}[t]\centering
\includegraphics[width=0.55\textwidth]{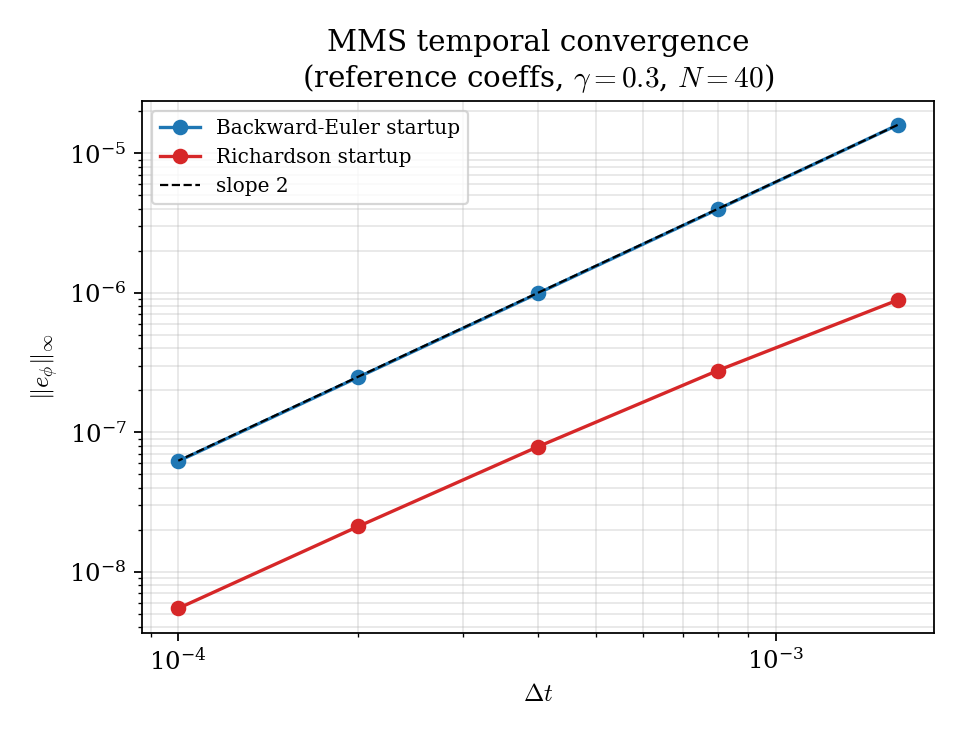}
\caption{Temporal convergence order measured by the method of manufactured
solutions (MMS), reference coefficients, both startups, against a
slope-$2$ reference.}\label{fig:conv}
\end{figure}

\paragraph{Agreement with an independent discretisation.} The two checks
above test the scheme's time-stepping order against itself or against a
forcing term built with the same code; a bug shared between the solver
and its own forcing-term construction would not be caught by either. As a
check on the underlying right-hand side itself, we therefore also compare
against an independently implemented nonlinear Newton--BDF2 solver that
treats $f'(\phi^{n+1})$ fully implicitly, with no SAV device at all. We run
both schemes from the same initial condition on a generic,
well-conditioned coefficient set,
\begin{gather*}
c(\phi)=0.4+0.2\cos\phi,\qquad
b(\phi)=\frac{c(\phi)^2}{1.2}+0.2,\qquad
A(\phi)=0.5+0.3\sin\phi, \\
\kappa(\phi)=0.6+0.3\phi^2,\qquad
f(\phi)=(\phi-0.3)^2(\phi-0.7)^2,
\end{gather*}
with $d_0=1.2$, $\gamma=0.3$, $\epsq=0.15$ (deliberately distinct from the
reference configuration of \Cref{tab:coef}, and free of its steep
$A(\phi)$ transition, to test that agreement is not special to that one
coefficient choice), on a $19\times19$ grid to $T=0.12$, at five step
sizes $\dt\in\{8,4,2,1,0.5\}\times10^{-3}$. The max-norm difference
between the two schemes' $\phi$ fields shrinks consistently at order
$2.04,2.02,2.00,2.01$ across all four successive halvings
(\Cref{fig:indep}), and the difference between their true energies shrinks
at a matching $O(\dt^2)$ rate, exactly as expected since both schemes are
independently second-order accurate approximations of the same dynamics,
so their pairwise difference should vanish at this same rate.

\begin{figure}[t]\centering
\includegraphics[width=0.55\textwidth]{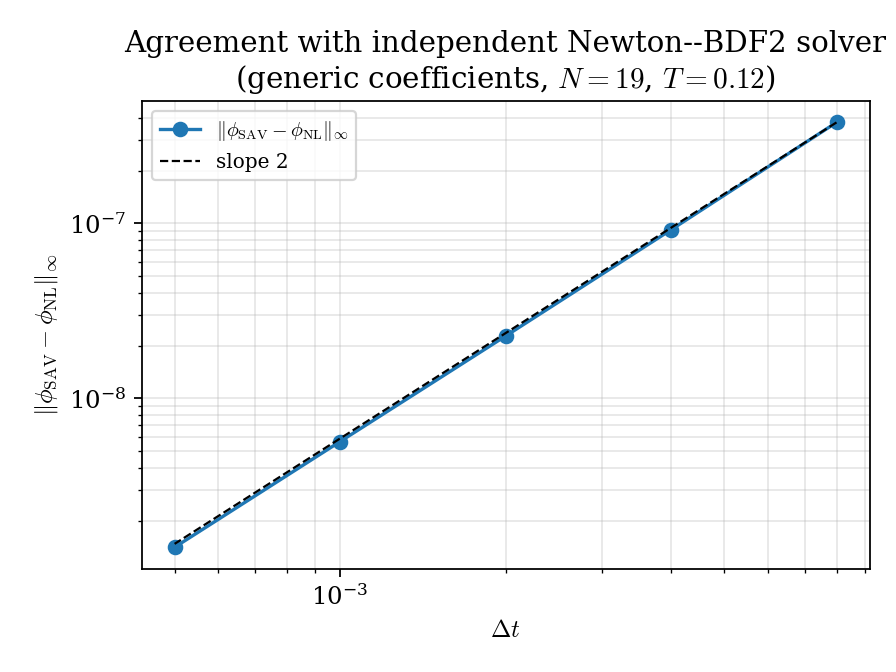}
\caption{Agreement with the independent nonlinear Newton--BDF2 solver,
generic coefficient set, $19\times19$ grid, $T=0.12$:
$\|\phi_{\mathrm{SAV}}-\phi_{\mathrm{NL}}\|_\infty$ against a slope-$2$
reference, all five $\dt$.}\label{fig:indep}
\end{figure}

\paragraph{Cost against the nonlinear scheme.}\label{sec:cost} At $\gamma=0.20$ (linearly
stable, so both schemes compute the same solution), the SAV scheme uses
$2.5$--$3.0\times$ fewer linear iterations and $1.7$--$2.0\times$ less wall
time than Newton--BDF2 at $N=64$--$128$, at substantially less memory (no
Jacobian storage). The underlying reason for this advantage is structural,
not just an empirical count: the linear SAV system is unconditionally
solvable at each step (\Cref{thm:solv}), whereas Newton additionally
requires its nonlinear iteration to actually converge; in the linearly
stable regime tested here this requirement costs little in practice
(Newton converges in 2--3 iterations even at $\dt\sim0.1$), but it is a
requirement the SAV scheme simply does not have. The simple Fourier
preconditioner does degrade at very large $\dt$ and at a steep $A(\phi)$.
The headline demonstration and physics study below, both run at this
steeper coefficient set, show this cost directly rather than avoiding it:
up to $845$ GMRES iterations per step at the coarsest $\dt$ tested
(\Cref{sec:headline}). This is a
preconditioner limitation (\eqref{eq:star} stays at round-off throughout),
underscoring that a multigrid preconditioner for this regime is a genuine
priority, not merely a theoretical concern, for future work.

\paragraph{Ablations.} Two parameters entering the scheme are, by
construction, not supposed to affect the computed solution, which we check here directly. First, \Cref{thm:main} only
requires $\epsq\ge0$, so the extra stress-smoothing term $\epsq\lap q$ is
not needed for stability. Sweeping $\epsq$ from $10^{-3}$ down to $0$
confirms that \eqref{eq:star} continues to hold and $\Emod$ remains
monotone at every value, including $\epsq=0$, with the computed solution
itself moving by only $4\times10^{-6}$. Second, the SAV constant $\CO$
(\Cref{sec:cont-energy}) is a free modelling choice introduced purely to
keep the square root well defined, and should not change the physical
solution: varying $\CO$ over five orders of magnitude changes $\phi,q$ by
at most $8\times10^{-11}$, confirming that it does not.

\subsection{Coupled-vs-uncoupled demonstration}\label{sec:physics}

Using the configuration above ($N=128$, $\dt=10^{-3}$, dealiased, identical
seeded IC), we run the coupled model and its $A\equiv0$ control (which
decouples $q$ and reduces \eqref{eq:chq-phi} to plain Cahn--Hilliard with the
same $b(\phi)$) to $T=10$.

\begin{figure}[t]\centering
\includegraphics[width=\textwidth]{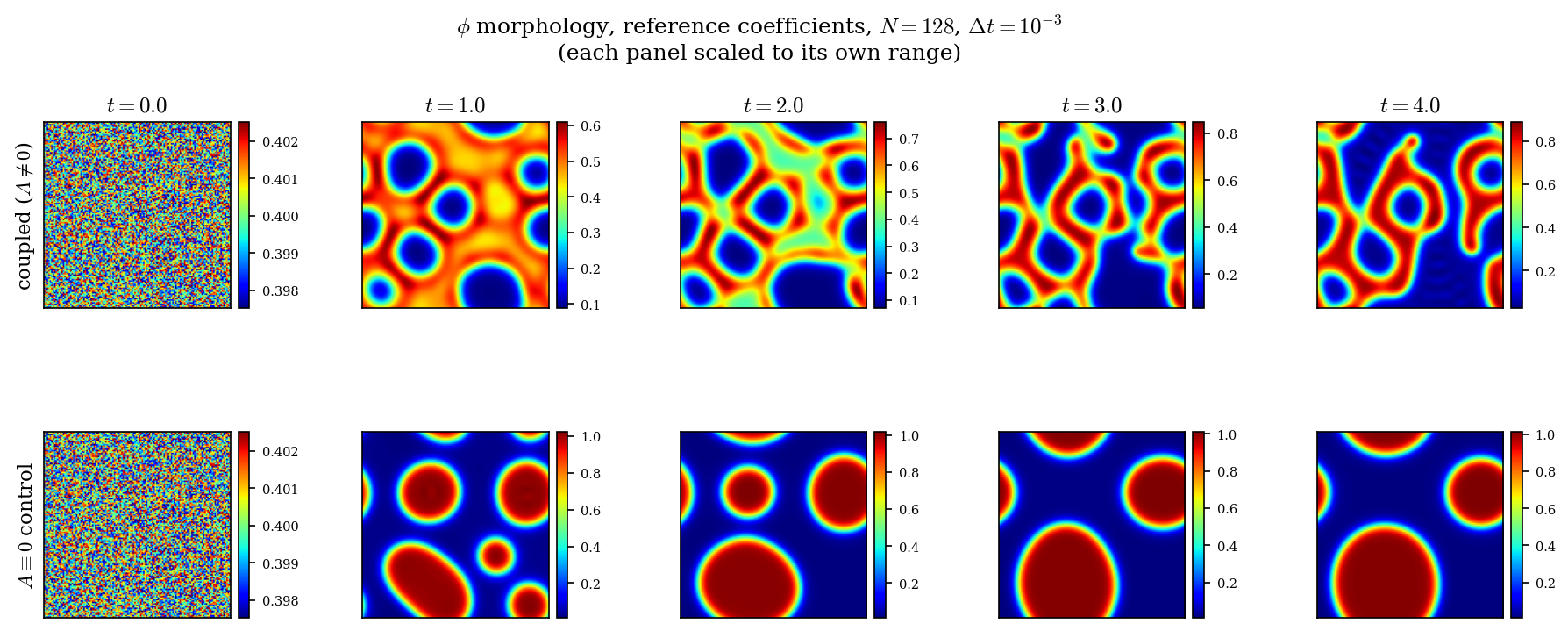}
\caption{$\phi$ morphology, reference coefficients, $N=128$,
$\dt=10^{-3}$, dealiased, at $t=0,1,2,3,4$ (the early, most rapidly evolving
part of the $T=10$ run); \emph{each panel is scaled to its own $\phi$ range}
(colour bars shown) so structure is visible at every time, including the
tiny initial perturbation at $t=0$. Top: coupled model. Bottom: $A\equiv0$
control (Cahn-Hilliard dynamics). The control already shows large, fully-saturated droplets by $t=1$
and coarsens by Ostwald ripening from there (plateauing by $t\approx7$--$8$,
beyond this figure's range); the coupled model instead organises much more
slowly into an interconnected network that stays at intermediate $\phi$
(never reaching the well minima) and is still visibly coarsening at $t=10$;
the two models are already qualitatively different within the first time
unit.}\label{fig:chqpy-morph}
\end{figure}

\begin{figure}[t]\centering
\includegraphics[width=0.7\textwidth]{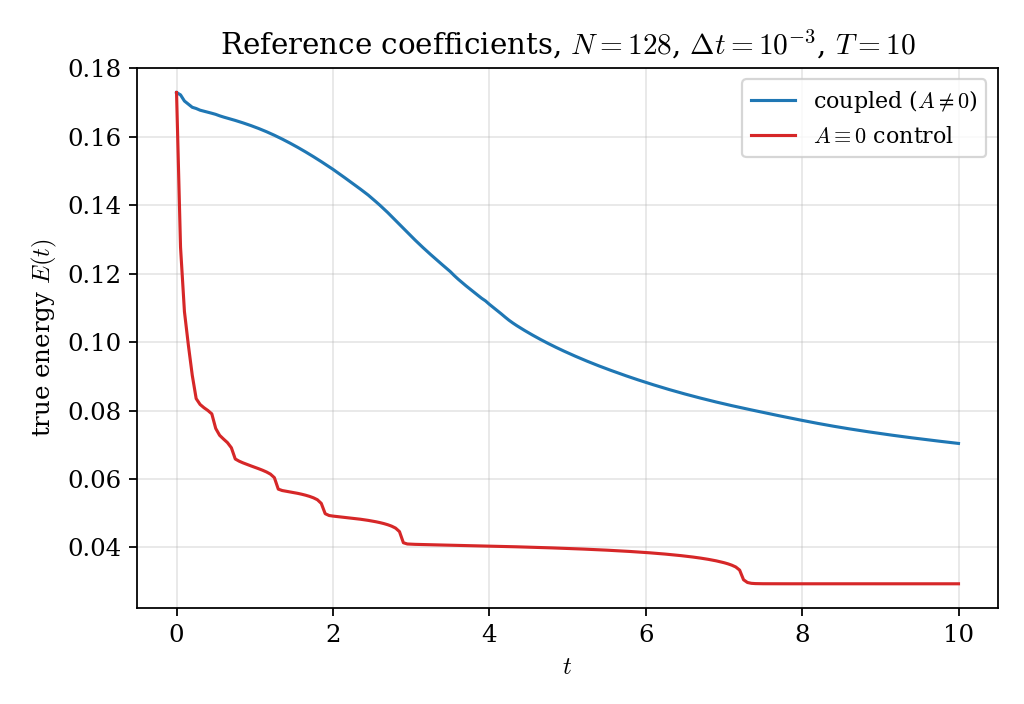}
\caption{True energy $E(t)$, same two runs. The $A\equiv0$ control falls
faster and plateaus by $t\approx7$--$8$ (visible step-drops from individual
droplet coalescence events); the coupled model keeps declining smoothly
throughout, not yet reaching equilibrium at $t=10$.}\label{fig:chqpy-energy}
\end{figure}

Mass is conserved to $\sim5\times10^{-12}$ in both runs and $\Emod$ decreases
strictly every step. The two models behave very differently
(\Cref{fig:chqpy-morph} and \Cref{fig:chqpy-energy}): from the identical
$E(0)=0.1729$, the $A\equiv0$ control has released $83\%$ of its initial
energy by $t=10$ and reached a coarsening plateau, while the coupled model
has released only $59\%$ and is still on a visibly downward slope, not yet
at equilibrium. The GMRES cost mirrors the morphology: the coupled model
needs a median $333$ iterations/step (the steep $A(\phi)$ term is genuinely
expensive to precondition, \Cref{sec:cost}), while the $A\equiv0$ control,
with no coupling term to resolve, needs only $16$.

\emph{Robustness}: the scheme remains stable over $10^4$ steps on
the reference configuration of the \Cref{tab:coef}, and the viscoelastic coupling
measurably slows and delays phase separation relative to an uncoupled
control from the identical initial perturbation, consistent with Tanaka's
viscoelastic suppression of phase separation \cite{Tanaka2000}. The
results above use a single realisation per case, at $\dt=10^{-3}$
specifically, given the cost of finer $\dt$ documented in
\Cref{sec:headline}. A systematic study of these dependencies is follow-on
work that this method was built to make affordable.

\section{Conclusion}

We have derived and verified a linear, second-order, unconditionally
energy-stable SAV--BDF2 scheme for the CHQ viscoelastic phase-separation
model. Stability rests on an exact discrete energy--dissipation identity
requiring only the model's own coercivity assumption, frozen at the
extrapolated state, with no time-step restriction. The energy--dissipation identity also holds for the post-step spectral-filtering variant used in the physics study of \Cref{sec:physics}, establishing the stability of the numerical method used to obtain these results. The per-step linear system is uniquely solvable for every $\dt>0$.
Numerically, the energy--dissipation identity holds to round-off at every time step, even for $\epsq=0$ and across five orders of magnitude in $C_0$. The method achieves second-order temporal convergence, including for the thin interface $\gamma=10^{-3}$, and spectral accuracy in space for the smooth test problems considered. Compared with an independent nonlinear Newton--BDF2 solver, it requires $2.5$--$3.0\times$ fewer linear iterations and $1.7$--$2.0\times$ less wall time. Moreover, it provides a per-step solvability guarantee that the nonlinear scheme lacks. On
the reference configuration of \Cref{sec:numerics}, run to $10^4$
steps from an identical initial perturbation, the method reproduces the
qualitative viscoelastic suppression of phase separation reported in the
literature: the coupled system forms a persistent, still-evolving gel-like
network that never fully separates into the pure phases, while an
uncoupled ($A\equiv0$) control instead forms discrete droplets that
coarsen to a fully phase-separated state. A
methodological note is itself relevant here: the scheme's un-dealiased
spectral discretisation can develop grid-scale aliasing noise from
evaluating nonlinear coefficients on the collocation grid; an optional
post-hoc spectral filter (provably stability-preserving, \Cref{rmk:filter})
removes it and is used throughout. A multigrid preconditioner for the local
$q$-block operator (whose cost this steeper coefficient set makes
concretely visible rather than hypothetical) and a rigorous
relative-energy error estimate in the style of \cite{Brunk2025} are left
for future work.

\section*{Declarations}

\paragraph{Acknowledgements.} We thank  M\'aria Luk\'acov\'a-Medvidov\'a
and  Aaron Brunk for helpful discussions.

\paragraph{Funding.} This work is funded by the Deutsche Forschungsgemeinschaft
(DFG, German Research Foundation) -- SFB 1551, Project-ID 464588647.

\paragraph{Competing interests.} The authors declare no competing interests.

\paragraph{Data and code availability.} The solver, verification scripts and
analysis code underlying every result in this paper are available from the
authors upon reasonable request. 

\paragraph{Use of AI-assisted technologies.} During the preparation of this work, the authors used ChatGPT and Claude to support code development, assist in formulating mathematical proofs, explore the relevant literature, and improve the clarity and readability of the manuscript. The authors critically reviewed the AI-assisted contributions and independently verified the mathematical arguments, proofs, and numerical results reported in this work. The authors take full responsibility for the accuracy and integrity of the content of this publication.

\appendix
\section{Proofs and derivations}\label{app:derivations}

This appendix gives a complete and self-contained details of the proof
and derivation that the main text presents only in brief: the exact
equivalence of the SAV reformulation, the proof of \Cref{prop:cont-energy},
the G-stability identity used in \Cref{thm:main}, the complete proof of
\Cref{thm:main} itself, the filtered-scheme stability extension of
\Cref{rmk:filter}, and the complete proof of \Cref{thm:solv}.

Almost every step below uses one
recurring trick, standard in the weak or variational
treatment of differential equations: rather than solving an equation like
$\pt\phi=(\text{some expression})$ pointwise, we multiply both sides by a
chosen function $g$ (a test function) and integrate over $\Om$,
writing the result as the inner product
$\inner{\pt\phi}{g}\eqdef\int_\Om(\pt\phi)\,g\,\dx$ (\Cref{sec:notation}).
Below, ``testing equation~(X) with $g$'' always means exactly this: multiply
equation~(X) by $g$ and integrate over $\Om$. Choosing $g$ well (often $g=\mu$, $g=q$, or the constant $g=1$) lets a
derivative move off the unknown and onto $g$ by integration by parts, with
no leftover boundary term, because $\Om$ is periodic and has no
boundary at all: these are identities \eqref{eq:ibp1} and \eqref{eq:ibp2}
of \Cref{sec:notation}. We prove them here, since they are used in
essentially every computation below. For any scalar field $g$ and vector
field $F$, the product rule for the divergence gives
\[
\dv(gF) = (\grad g)\cdot F + g\,(\dv F) ,
\qquad\text{i.e.}\qquad
g\,(\dv F) = \dv(gF) - (\grad g)\cdot F .
\]
Integrating both sides over $\Om$,
\[
\int_\Om g\,(\dv F)\,\dx = \int_\Om \dv(gF)\,\dx - \int_\Om(\grad g)\cdot F\,\dx .
\]

By the divergence theorem,
\(
\int_\Om \dv(gF),\dx
=\oint_{\partial\Om}(gF)\cdot n,\mathrm dS,
\)
where \(n\) denotes the outward unit normal to \(\partial\Om\). For a periodic domain \(\Om\), the boundary integral vanishes because opposite faces are identified. By periodicity, the values of \(g\) and \(F\) agree at corresponding points on opposite faces. However, the outward unit normals at these points have opposite directions. Consequently, the corresponding flux contributions cancel pairwise, yielding
\(
\oint_{\partial\Om}(gF)\cdot n,\mathrm dS=0.
\)
Thus, the integral of \(\dv(gF)\) over \(\Om\) vanishes for any periodic fields \(g\) and \(F\). Hence
\[
\int_\Om g\,(\dv F)\,\dx = -\int_\Om(\grad g)\cdot F\,\dx ,
\qquad\text{i.e.}\qquad
\inner{\dv F}{g}=-\inner{F}{\grad g} ,
\]
which is exactly \eqref{eq:ibp1}. Identity \eqref{eq:ibp2} follows at once
by taking $F=\grad u$ for a scalar field $u$, since then
$\dv F=\dv(\grad u)=\lap u$:
\[
\inner{\lap u}{g}=\inner{\dv(\grad u)}{g}=-\inner{\grad u}{\grad g} .
\]
Both are repeated here for convenience:
\[
\inner{\dv F}{g}=-\inner{F}{\grad g},
\qquad
\inner{\lap u}{g}=-\inner{\grad u}{\grad g} .
\]

A second recurring convention concerns the star subscript $(\bs,\cs,\As,\ks,\Gs)$, which indicates that a coefficient is evaluated at the extrapolated, already-known state $\phi^\star$ or its clamped version $\phi^\star_b$ (see \Cref{sec:scheme}). These coefficients are independent of the unknown $\phi^{n+1}$ and are therefore fixed and known when the equations are solved. This property ensures that the scheme is linear in $(\phi^{n+1},\mu^{n+1},q^{n+1},r^{n+1})$. It also allows these coefficients to be moved in and out of inner products as appropriate. Finally, operators $\Dtwo$ and $\dd$ (see \eqref{eq:bdf2-ops}) are time-difference operators involving three consecutive time levels of a quantity, rather than spatial derivatives. Finally, $\Phimix(u,v)\eqdef\norm u^2+\norm{2u-v}^2$ is a specific combination of two time levels, introduced because it yields a telescoping identity under the BDF2 recursion, as shown in \eqref{eq:G} below. It has no independent significance beyond this role.

\paragraph{Exact equivalence of the SAV reformulation.} The
SAV trick replaces the nonlinear term $f'(\phi)$ in \eqref{eq:chq-mu} by
the product $r\,G(\phi)$ of a new unknown $r(t)$ and a known function $G$,
and supplies $r$ with its own evolution equation. This is only useful if
the enlarged system has exactly the same solutions as the original
one. We show this
by proving that, along any solution of the enlarged system, $r(t)$ is
forced to always equal $\sqrt{\Ef(\phi(t))+\CO}$, which is exactly the
value that makes $r\,G(\phi)$ equal $f'(\phi)$.

With $r(t)\eqdef\sqrt{\Ef(\phi(t))+\CO}$ and $G(\phi)\eqdef f'(\phi)/\sqrt{\Ef(\phi)+\CO}$ (so, by construction,
$r(t)G(\phi(t))=f'(\phi(t))$ whenever $r(t)$ happens to take this
particular value), the SAV system replaces \eqref{eq:chq-mu} by
$\mu=-\gamma\lap\phi+rG(\phi)$, where $r$ is now an independent unknown,
and appends the equation $\dot r=\half\inner{G(\phi)}{\pt\phi}$ to
determine it. We must check that this new equation for $r$ is actually
consistent with the formula that motivated it. By the chain rule,
\[
\frac{\mathrm d}{\mathrm dt}\sqrt{\Ef(\phi(t))+\CO}
=\frac{1}{2\sqrt{\Ef(\phi(t))+\CO}}\inner{f'(\phi)}{\pt\phi} .
\]
If, at some instant, $r=\sqrt{\Ef(\phi)+\CO}$ already holds, then also
$f'(\phi)=rG(\phi)$ at that instant (by definition of $G$), so the
right-hand side above becomes exactly $\half\inner{G(\phi)}{\pt\phi}$,
precisely the right-hand side of the SAV system's own equation for $\dot
r$. In other words: the function $t\mapsto\sqrt{\Ef(\phi(t))+\CO}$ and any
solution $r(t)$ of the SAV system satisfy the same differential
equation for $\dot r$ (with $\phi(t)$ regarded as already determined), and
they start from the same value at $t=0$, since $r(0)$ is initialised as
$\sqrt{\Ef(\phi(0))+\CO}$.

Two functions of $t$ with the same derivative at every instant and the
same value at one instant must be equal for all time: their difference has
derivative zero everywhere and value zero at $t=0$, so by the fundamental
theorem of calculus the difference is zero everywhere. Hence
$r(t)\equiv\sqrt{\Ef(\phi(t))+\CO}$ along any solution of the SAV system,
and therefore $r(t)G(\phi(t))\equiv f'(\phi(t))$ for all $t$: the SAV
$\mu$-equation reduces identically to the original \eqref{eq:chq-mu}, so
every solution of the SAV system solves \eqref{eq:chq}. Conversely, given
any solution of \eqref{eq:chq}, defining $r(t)\eqdef\sqrt{\Ef(\phi(t))+\CO}$
satisfies the SAV system's $r$-equation automatically, by the same
chain-rule computation run in reverse: the two systems have exactly the
same solutions.

\paragraph{Proof of \Cref{prop:cont-energy}.} We compute
$\frac{\mathrm d}{\mathrm dt}E$ directly from the definition of $E$
\eqref{eq:energy}, substitute the two evolution equations
\eqref{eq:chq-phi} and \eqref{eq:chq-q} to eliminate $\pt\phi,\pt q$, and
then complete the square in the resulting cross term to reveal a
manifestly non-negative quantity. 

\emph{Step 1: Differentiate $E$ term by term.}

Recall that
$E=\int_\Om\bigl[\tfrac\gamma2|\grad\phi|^2+f(\phi)+\tfrac12q^2\bigr]\dx$.
We differentiate each term in the integrand with respect to time. For smooth solutions, differentiation and integration over the fixed domain $\Om$ commute. For the first term, we use $|\grad\phi|^2=\grad\phi\cdot\grad\phi$. The product rule then gives
$\tfrac{\mathrm d}{\mathrm dt}|\grad\phi|^2=2,\grad\phi\cdot\grad\pt\phi$,
where we have used the commutativity of time and spatial derivatives, $\pt(\grad\phi)=\grad(\pt\phi)$. Multiplication by the prefactor $\tfrac\gamma2$ yields $\gamma,\grad\phi\cdot\grad\pt\phi$. For the second term, the chain rule gives $\tfrac{\mathrm d}{\mathrm dt}f(\phi)=f'(\phi),\pt\phi$. For the third term, we obtain $\tfrac{\mathrm d}{\mathrm dt}\bigl(\tfrac12q^2\bigr)=q \pt q$. Integrating these pointwise derivatives over $\Om$ and expressing the resulting integrals as inner products,

\begin{equation}
\label{eq:derivative_E}
\frac{\mathrm d}{\mathrm dt}E
=\gamma\inner{\grad\phi}{\grad\pt\phi}+\inner{f'(\phi)}{\pt\phi}+\inner{q}{\pt q} .
\end{equation}

\emph{Step 2: Combine the first two terms into $\inner{\mu}{\pt\phi}$.}
Identity \eqref{eq:ibp2}, $\inner{\lap u}{g}=-\inner{\grad u}{\grad g}$, rearranges
(swap the two sides and negate) to $\inner{\grad u}{\grad g}=-\inner{\lap u}{g}$ and
applying this with $u=\phi$, $g=\pt\phi$, we get
$\inner{\grad\phi}{\grad\pt\phi}=-\inner{\lap\phi}{\pt\phi}$. Multiplying by $\gamma$ and using the bilinearity of the inner product, we obtain
$\gamma\inner{\grad\phi}{\grad\pt\phi}=-\gamma\inner{\lap\phi}{\pt\phi}=\inner{-\gamma\lap\phi}{\pt\phi}$. Adding the second term in \eqref{eq:derivative_E} and
using the additivity of the inner product, $\inner{a}{c}+\inner{b}{c}=\inner{a+b}{c}$ with
$a=-\gamma\lap\phi$, $b=f'(\phi)$, $c=\pt\phi$, we obtain
\[
\gamma\inner{\grad\phi}{\grad\pt\phi}+\inner{f'(\phi)}{\pt\phi}
=\inner{-\gamma\lap\phi+f'(\phi)}{\pt\phi} .
\]
By \eqref{eq:chq-mu}, $-\gamma\lap\phi+f'(\phi)=\mu$, so this
equals $\inner{\mu}{\pt\phi}$, and Step 1 becomes
\[
\frac{\mathrm d}{\mathrm dt}E=\inner{\mu}{\pt\phi}+\inner{q}{\pt q} .
\]

\emph{Step 3: Substitute $\pt\phi$ from \eqref{eq:chq-phi}.} Equation
\eqref{eq:chq-phi} says $\pt\phi=\dv(b\grad\mu-c\grad(Aq))$. Substituting
this expression for $\pt\phi$ directly into $\inner{\mu}{\pt\phi}$, we get
\[
\inner{\mu}{\pt\phi}=\inner{\mu}{\dv(b\grad\mu-c\grad(Aq))} .
\]
By the symmetry of the inner product, $\inner{\mu}{\dv F}=\inner{\dv F}{\mu}$
for any flux $F$ and applying \eqref{eq:ibp1}, $\inner{\dv F}{g}=-\inner{F}{\grad g}$,
with $F=b\grad\mu-c\grad(Aq)$ and $g=\mu$, we obtain
\[
\inner{\mu}{\pt\phi}=-\inner{b\grad\mu-c\grad(Aq)}{\grad\mu} .
\]

By the bilinearity of the inner product, we can expand the flux term as
$\inner{b\grad\mu-c\grad(Aq)}{\grad\mu}
=\inner{b\grad\mu}{\grad\mu}-\inner{c\grad(Aq)}{\grad\mu}$.
Consequently,
\[
\inner{\mu}{\pt\phi}=-\inner{b\grad\mu}{\grad\mu}+\inner{c\grad(Aq)}{\grad\mu}
=-\int_\Om b|\grad\mu|^2+\int_\Om c\,\grad\mu\cdot\grad(Aq) ,
\]

using $\inner{b\grad\mu}{\grad\mu}=\int_\Om b\,\grad\mu\cdot\grad\mu
=\int_\Om b|\grad\mu|^2$ and the commutativity of the dot product,
$\grad(Aq)\cdot\grad\mu=\grad\mu\cdot\grad(Aq)$, inside the second
integral.

\emph{Step 4: Substitute $\pt q$ from \eqref{eq:chq-q}.} Equation
\eqref{eq:chq-q} says
$\pt q=-\kappa q-A\,\dv F+\epsq\lap q$, with the flux $F\eqdef c\grad\mu-d_0\grad(Aq)$. Substituting into $\inner{q}{\pt q}$ (testing
\eqref{eq:chq-q} with $g=q$) and distributing the inner product over the
three terms of the sum (bilinearity in the second slot), we get
\[
\inner{q}{\pt q}=\inner{q}{-\kappa q}+\inner{q}{-A\,\dv F}+\inner{q}{\epsq\lap q}
=-\inner{\kappa q}{q}-\inner{A\,\dv F}{q}+\epsq\inner{\lap q}{q} .
\]
We can now consider the three terms on the right-hand side separately:
\begin{itemize}
\item The first is already a plain integral:
$-\inner{\kappa q}{q}=-\int_\Om\kappa q^2$.
\item The third, by \eqref{eq:ibp2} with $u=q,g=q$: $\inner{\lap q}{q}=-\inner{\grad q}{\grad q}=-\norm{\grad q}^2$,
so $\epsq\inner{\lap q}{q}=-\epsq\norm{\grad q}^2$.
\item The second needs the same manoeuvre used repeatedly in
\Cref{sec:accuracy}--\Cref{sec:numerics}. At this point in the argument,
$A=A(\phi)$ is simply a known scalar field: it depends on neither $\mu$
nor the unknown being tested against. So the product $Aq$ is itself a
perfectly good test function.

Concretely,
\[
\inner{A\,\dv F}{q}=\int_\Om A\,(\dv F)\,q\,\dx
=\int_\Om(\dv F)\,(Aq)\,\dx=\inner{\dv F}{Aq} ,
\]
since the three scalar factors $A,\dv F,q$ can be regrouped freely under
the integral: ordinary multiplication of numbers is commutative and
associative. Applying \eqref{eq:ibp1} with $g=Aq$ gives
$\inner{\dv F}{Aq}=-\inner{F}{\grad(Aq)}$, so
$\inner{A\,\dv F}{q}=-\inner{F}{\grad(Aq)}$.

Substituting $F=c\grad\mu-d_0\grad(Aq)$ and distributing the inner product
over its two terms,
\begin{align*}
-\inner{A\,\dv F}{q}
&=\inner{F}{\grad(Aq)}
=\inner{c\grad\mu-d_0\grad(Aq)}{\grad(Aq)}\\
&=\inner{c\grad\mu}{\grad(Aq)}-d_0\inner{\grad(Aq)}{\grad(Aq)}
=\int_\Om c\,\grad\mu\cdot\grad(Aq)-d_0\norm{\grad(Aq)}^2 .
\end{align*}

\end{itemize}
Collecting the three pieces, we find
\[
\inner{q}{\pt q}=-\int_\Om\kappa q^2+\int_\Om c\,\grad\mu\cdot\grad(Aq)
-d_0\int_\Om|\grad(Aq)|^2-\epsq\norm{\grad q}^2 .
\]

\emph{Step 5: Add Steps 3 and 4.} Writing the two results one under the
other,
\[
\begin{aligned}
\inner{\mu}{\pt\phi}&=-\int_\Om b|\grad\mu|^2+\int_\Om c\,\grad\mu\cdot\grad(Aq) ,\\
\inner{q}{\pt q}&=\phantom{-\int_\Om b|\grad\mu|^2+{}}\int_\Om c\,\grad\mu\cdot\grad(Aq)
-d_0\int_\Om|\grad(Aq)|^2-\int_\Om\kappa q^2-\epsq\norm{\grad q}^2 ,
\end{aligned}
\]
and adding column by column, we obtain
\[
\frac{\mathrm d}{\mathrm dt}E=\inner{\mu}{\pt\phi}+\inner{q}{\pt q}
=-\int_\Om b|\grad\mu|^2+2\int_\Om c\,\grad\mu\cdot\grad(Aq)
-d_0\int_\Om|\grad(Aq)|^2-\int_\Om\kappa q^2-\epsq\norm{\grad q}^2 .
\]
This is the point where the sign of the coupling term in
\eqref{eq:chq-phi} matters: it is chosen precisely so that testing both
equations produces the same cross term twice, rather than two cross
terms that would instead cancel each other (in which case the calculation
below would not produce a sign-definite dissipation at all, and the
coercivity condition \eqref{eq:coerc-cont} discussed in
\Cref{sec:cont-energy} above would not arise).

\emph{Step 6: complete the square.} The first three terms on the right
have the pointwise form $-b|u|^2+2c(u\cdot w)-d_0|w|^2$ with $u=\grad\mu$,
$w=\grad(Aq)$, exactly the configuration of identity \eqref{eq:CS},
stated above in the proof of \Cref{prop:cont-energy}. The algebraic identity \eqref{eq:CS} 

holds because expanding its right-hand side gives
$\bigl(b-\tfrac{c^2}{d_0}\bigr)|u|^2
+\tfrac1{d_0}\bigl(c^2|u|^2-2cd_0(u\cdot w)+d_0^2|w|^2\bigr)
=b|u|^2-\tfrac{c^2}{d_0}|u|^2+\tfrac{c^2}{d_0}|u|^2-2c(u\cdot w)+d_0|w|^2
=b|u|^2-2c(u\cdot w)+d_0|w|^2$, since the two $\tfrac{c^2}{d_0}|u|^2$ terms
(one negative, one positive) cancel exactly. Negating both sides of this
identity converts $-b|u|^2+2c(u\cdot w)-d_0|w|^2$ pointwise, at every
$x\in\Om$, into $-\bigl(b-\tfrac{c^2}{d_0}\bigr)|u|^2-\tfrac1{d_0}|cu-d_0w|^2$.

\emph{Step 7: assemble $\Dcont$.} Substituting this pointwise rewriting
into the first three terms of the result of Step 5, and integrating over
$\Om$, gives
\[
\frac{\mathrm d}{\mathrm dt}E
=-\int_\Om\Bigl[\Bigl(b-\tfrac{c^2}{d_0}\Bigr)|\grad\mu|^2
+\tfrac1{d_0}|c\grad\mu-d_0\grad(Aq)|^2\Bigr]\dx
-\int_\Om\kappa q^2-\epsq\norm{\grad q}^2 .
\]
Collecting all four terms under one integral, with an overall minus sign
in front, reproduces exactly $\frac{\mathrm d}{\mathrm dt}E=-\Dcont$, with
$\Dcont$ as defined in \eqref{eq:cont-D}. Term by term, the four summands
match: $\bigl(b-\tfrac{c^2}{d_0}\bigr)|\grad\mu|^2$,
$\tfrac1{d_0}|c\grad\mu-d_0\grad(Aq)|^2$, $\kappa q^2$, and
$\epsq|\grad q|^2$ are precisely the four terms inside the integral
defining $\Dcont$.

\paragraph{Proof of the G-stability identity \eqref{eq:G}.} This identity is not new. It is the Nevanlinna--Odeh multiplier for BDF2
\cite{NevanlinnaOdeh1981}, the standard tool for proving G-stability
\cite{Dahlquist1976} of multistep methods. We include its proof because it
is used repeatedly above. The identity is purely algebraic: it holds for
any three elements $x,y,z$ of the same space and does not use the model
equations. Its role is to rewrite a term of the form
$\inner{3x-4y+z}{x}$, which arises whenever a BDF2 formula is tested
against its newest level, as a telescoping difference of $\Phimix$ plus a
non-negative remainder. This telescoping turns the step-by-step
computation into a statement that the energy decreases monotonically. The
identity is verified by expanding both sides into inner products and
comparing coefficients.

Let $x=a^{n+1}$, $y=a^n$, $z=a^{n-1}$, where $a$ denotes any of the
quantities $\grad\phi$, $q$, $r$. Both sides of \eqref{eq:G} can be written
in terms of the six quantities
$\norm x^2,\inner xy,\norm y^2,\inner yz,\norm z^2,\inner xz$.
Using $\norm{u-v}^2=\norm u^2-2\inner uv+\norm v^2$, we obtain
\begin{align*}
\norm{2x-y}^2&=4\norm x^2-4\inner xy+\norm y^2 ,\\
\norm{2y-z}^2&=4\norm y^2-4\inner yz+\norm z^2 ,\\
\norm{x-2y+z}^2&=\norm x^2+4\norm y^2+\norm z^2-4\inner xy+2\inner xz-4\inner yz .
\end{align*}
The last line follows by expanding $\inner{x-2y+z}{x-2y+z}$: the squared
terms give $\norm x^2+4\norm y^2+\norm z^2$, and the cross terms give
$2\inner{x}{-2y}+2\inner{x}{z}+2\inner{-2y}{z}
=-4\inner xy+2\inner xz-4\inner yz$. Inserting these expressions into the
right-hand side of \eqref{eq:G} yields
\begin{align*}
\half\bigl[\Phimix(x,y)-\Phimix(y,z)\bigr]+\half\norm{x-2y+z}^2
&=\half\bigl[(\norm x^2+\norm{2x-y}^2)-(\norm y^2+\norm{2y-z}^2)\bigr]\\
&\quad+\half\norm{x-2y+z}^2\\
&=\half\bigl[5\norm x^2-4\inner xy-4\norm y^2+4\inner yz-\norm z^2\bigr]\\
&+\half\bigl[\norm x^2+4\norm y^2+\norm z^2-4\inner xy+2\inner xz-4\inner yz\bigr]\\
&=3\norm x^2-4\inner xy+\inner xz .
\end{align*}

In the last step, all terms involving $\norm y^2$, $\inner yz$ and
$\norm z^2$ cancel between the two brackets. On the other hand, by the
symmetry of the inner product,
\[
\inner{3x-4y+z}{x}=3\norm x^2-4\inner xy+\inner xz .
\]
The two sides of \eqref{eq:G} therefore coincide, which proves the identity.

\paragraph{Proof of \Cref{thm:main}.} The strategy mirrors the continuous proof of \Cref{prop:cont-energy}. It
consists of three steps.
\begin{enumerate}[(i)]
\item \emph{Testing.} We test \eqref{eq:S1} with $\mu^{n+1}$,
\eqref{eq:S2} with $\Dtwo\phi^{n+1}$ and \eqref{eq:S3} with $q^{n+1}$. The
mixed term $\int_\Om \cs\,\grad\mu^{n+1}\cdot\grad(\As q^{n+1})$ then
appears in both the $\phi$- and the $q$-equation. Adding the two results
yields the same doubled cross term as in the continuous proof.
\item \emph{Dissipation.} We apply the pointwise identity \eqref{eq:CS}
to this expression. It turns the sum of the spatial terms into
$-\Dmod^{n+1}$, which is non-negative.
\item \emph{Telescoping.} We apply the algebraic G-stability identity
\eqref{eq:G} to each BDF2 term $\inner{\Dtwo a^{n+1}}{a^{n+1}}$ with
$a=\grad\phi,q,r$. Each term becomes a difference of consecutive values of
$\Phimix$, plus a non-negative remainder.
\end{enumerate} Collecting the three resulting telescoping
differences reproduces $\Emod^{n+1}-\Emod^n$, giving exactly \eqref{eq:star}.

\emph{Step 1: Eliminate $\mu^{n+1}$ between \eqref{eq:S1} and
\eqref{eq:S2}.} Take the inner product of \eqref{eq:S2}
($\mu^{n+1}=-\gamma\lap\phi^{n+1}+r^{n+1}\Gs$) with $\Dtwo\phi^{n+1}$, i.e.\
multiply both sides of \eqref{eq:S2} by $\Dtwo\phi^{n+1}$ and
integrate over $\Om$:
\[
\inner{\mu^{n+1}}{\Dtwo\phi^{n+1}}
=-\gamma\inner{\lap\phi^{n+1}}{\Dtwo\phi^{n+1}}+r^{n+1}\inner{\Gs}{\Dtwo\phi^{n+1}} .
\]
By \eqref{eq:ibp2}, we have $-\gamma\inner{\lap\phi^{n+1}}{\Dtwo\phi^{n+1}}
=\gamma\inner{\grad\phi^{n+1}}{\grad\Dtwo\phi^{n+1}}$ and since $\Dtwo$ is
only a fixed linear combination of a field's values at three time levels
(it never touches the spatial variable), it commutes with the spatial
gradient, $\grad\Dtwo\phi^{n+1}=\Dtwo\grad\phi^{n+1}$. This is the
commutation hypothesis of \Cref{thm:main}. So
\[
\inner{\mu^{n+1}}{\Dtwo\phi^{n+1}}
=\gamma\inner{\grad\phi^{n+1}}{\Dtwo\grad\phi^{n+1}}+r^{n+1}\inner{\Gs}{\Dtwo\phi^{n+1}} .
\]
Equation \eqref{eq:S4} ($\Dtwo r^{n+1}=\half\inner{\Gs}{\Dtwo\phi^{n+1}}$),
multiplied by $2$, says $\inner{\Gs}{\Dtwo\phi^{n+1}}=2\Dtwo r^{n+1}$, so
\begin{equation}\label{eq:app-a}
\inner{\mu^{n+1}}{\Dtwo\phi^{n+1}}
= \gamma\inner{\grad\phi^{n+1}}{\Dtwo\grad\phi^{n+1}}+2r^{n+1}\Dtwo r^{n+1} .
\end{equation}
Separately, take the inner product of \eqref{eq:S1} with $\mu^{n+1}$, i.e.\
test \eqref{eq:S1} with $g=\mu^{n+1}$, and use \eqref{eq:ibp1} to get
\begin{equation}\label{eq:app-b}
\inner{\Dtwo\phi^{n+1}}{\mu^{n+1}}
=-\int_\Om\bs|\grad\mu^{n+1}|^2+\int_\Om\cs\grad(\As q^{n+1})\cdot\grad\mu^{n+1} .
\end{equation}
The left-hand sides of \eqref{eq:app-a} and \eqref{eq:app-b} are
equal by the symmetry of the inner product. So, equating the right-hand sides of \eqref{eq:app-a} and \eqref{eq:app-b} yields
\begin{equation}\label{eq:app-I}
\gamma\inner{\grad\phi^{n+1}}{\Dtwo\grad\phi^{n+1}}+2r^{n+1}\Dtwo r^{n+1}
=-\int_\Om\bs|\grad\mu^{n+1}|^2+\int_\Om\cs\grad(\As q^{n+1})\cdot\grad\mu^{n+1} .
\end{equation}
This single equation has eliminated $\mu^{n+1}$ from appearing on its own and it enters only through $\grad\mu^{n+1}$ on the right, 
exactly as in the continuous proof above.

\emph{Step 2: Bring in the $q$-equation.} Take the inner product of
\eqref{eq:S3} with $q^{n+1}$ (test \eqref{eq:S3} with $g=q^{n+1}$). Write
$F\eqdef\cs\grad\mu^{n+1}-d_0\grad(\As q^{n+1})$ for the flux inside the
divergence in \eqref{eq:S3}, so that \eqref{eq:S3} reads
$\Dtwo q^{n+1}=-\ks q^{n+1}-\As\dv F+\epsq\lap q^{n+1}$. Distributing the
inner product over the three terms on the right (bilinearity in the
second slot),
\[
\inner{\Dtwo q^{n+1}}{q^{n+1}}
=-\inner{\ks q^{n+1}}{q^{n+1}}-\inner{\As\dv F}{q^{n+1}}
+\epsq\inner{\lap q^{n+1}}{q^{n+1}} .
\]
Again, we consider the three terms on the right-hand side separately:
\begin{itemize}
\item The first is already a plain integral:
$-\inner{\ks q^{n+1}}{q^{n+1}}=-\int_\Om\ks(q^{n+1})^2$.
\item The third, by \eqref{eq:ibp2} with $u=g=q^{n+1}$:
$\inner{\lap q^{n+1}}{q^{n+1}}=-\inner{\grad q^{n+1}}{\grad q^{n+1}}
=-\norm{\grad q^{n+1}}^2$, so
$\epsq\inner{\lap q^{n+1}}{q^{n+1}}=-\epsq\norm{\grad q^{n+1}}^2$.
\item For the second term, we use the same regrouping as in the proof of
\Cref{prop:cont-energy}. Since $\As$ is a known function of $x$ (it is
frozen), $\As q^{n+1}$ is an admissible test function. Indeed,
\[
\inner{\As\dv F}{q^{n+1}}=\int_\Om\As(\dv F)\,q^{n+1}\dx
=\int_\Om(\dv F)\,(\As q^{n+1})\dx=\inner{\dv F}{\As q^{n+1}} .
\]
Applying \eqref{eq:ibp1} with $g=\As q^{n+1}$ gives
\[
\inner{\As\dv F}{q^{n+1}}=-\inner{F}{\grad(\As q^{n+1})} .
\]
With $F=\cs\grad\mu^{n+1}-d_0\grad(\As q^{n+1})$, this yields
\begin{align*}
-\inner{\As\dv F}{q^{n+1}}&=\inner{F}{\grad(\As q^{n+1})}\\
&=\int_\Om\cs\grad\mu^{n+1}\cdot\grad(\As q^{n+1})
 -d_0\int_\Om|\grad(\As q^{n+1})|^2 .
\end{align*}

\end{itemize}
Collecting the three pieces,
\begin{equation}\label{eq:app-II}
\inner{\Dtwo q^{n+1}}{q^{n+1}}
=-\int_\Om\ks(q^{n+1})^2+\int_\Om\cs\grad(\As q^{n+1})\cdot\grad\mu^{n+1}
-d_0\int_\Om|\grad(\As q^{n+1})|^2-\epsq\norm{\grad q^{n+1}}^2 .
\end{equation}
The term $\int_\Om\cs\grad(\As q^{n+1})\cdot\grad\mu^{n+1}$ appears in
both \eqref{eq:app-I} and \eqref{eq:app-II}. It is the discrete
counterpart of the cross term that appears twice in the proof of
\Cref{prop:cont-energy}. The reason is the same: \eqref{eq:S1} and
\eqref{eq:S3} contain the same flux $\cs\grad\mu^{n+1}-\ldots$ with the
same relative sign, which is fixed by the sign convention in
\eqref{eq:chq-phi}. As in Step 5 of the proof of \Cref{prop:cont-energy},
we write \eqref{eq:app-I} and \eqref{eq:app-II} one under the other and
add them:
\begin{align*}
\gamma\inner{\grad\phi^{n+1}}{\Dtwo\grad\phi^{n+1}}+2r^{n+1}\Dtwo r^{n+1}
&=-\int_\Om\bs|\grad\mu^{n+1}|^2+\int_\Om\cs\grad(\As q^{n+1})\cdot\grad\mu^{n+1} ,\\
\inner{\Dtwo q^{n+1}}{q^{n+1}}
&=\int_\Om\cs\grad(\As q^{n+1})\cdot\grad\mu^{n+1}
-d_0\int_\Om|\grad(\As q^{n+1})|^2\\
&\quad-\int_\Om\ks(q^{n+1})^2-\epsq\norm{\grad q^{n+1}}^2 ,
\end{align*}
the shared term occurs once in each line, so adding the two equations does
not cancel it but doubles it:
\begin{multline}\label{eq:app-III}
\gamma\inner{\grad\phi^{n+1}}{\Dtwo\grad\phi^{n+1}}+2r^{n+1}\Dtwo r^{n+1}
+\inner{\Dtwo q^{n+1}}{q^{n+1}}\\
=-\int_\Om\bs|\grad\mu^{n+1}|^2+2\int_\Om\cs\grad(\As q^{n+1})\cdot\grad\mu^{n+1}\\
-d_0\int_\Om|\grad(\As q^{n+1})|^2-\int_\Om\ks(q^{n+1})^2-\epsq\norm{\grad q^{n+1}}^2 .
\end{multline}

\emph{Step 3: Complete the square.} The first three terms on the right of
\eqref{eq:app-III} are exactly $-\bs|u|^2+2\cs(u\cdot w)-d_0|w|^2$ with
$u=\grad\mu^{n+1}$, $w=\grad(\As q^{n+1})$, pointwise. By the identity
\eqref{eq:CS} (derived explicitly above, in the proof of
\Cref{prop:cont-energy}) with $b=\bs,c=\cs$,
\[
\bs|u|^2-2\cs(u\cdot w)+d_0|w|^2=\Bigl(\bs-\tfrac{\cs^2}{d_0}\Bigr)|u|^2
+\tfrac1{d_0}|\cs u-d_0w|^2 ,
\]
so negating both sides converts the first three terms of
\eqref{eq:app-III}, pointwise at every $x\in\Om$, into
$-\bigl(\bs-\tfrac{\cs^2}{d_0}\bigr)|u|^2-\tfrac1{d_0}|\cs u-d_0w|^2$.
Substituting this into \eqref{eq:app-III} and integrating over $\Om$,
\begin{multline*}
\gamma\inner{\grad\phi^{n+1}}{\Dtwo\grad\phi^{n+1}}+2r^{n+1}\Dtwo r^{n+1}
+\inner{\Dtwo q^{n+1}}{q^{n+1}}\\
=-\int_\Om\Bigl[\Bigl(\bs-\tfrac{\cs^2}{d_0}\Bigr)|\grad\mu^{n+1}|^2
+\tfrac1{d_0}|\cs\grad\mu^{n+1}-d_0\grad(\As q^{n+1})|^2\Bigr]\dx\\
-\int_\Om\ks(q^{n+1})^2-\epsq\norm{\grad q^{n+1}}^2 .
\end{multline*}
Collecting the four terms on the right-hand side under one integral, with
an overall minus sign, gives exactly $-\Dmod^{n+1}$ as defined in
\eqref{eq:Dmod}. The four integrands
$\bigl(\bs-\tfrac{\cs^2}{d_0}\bigr)|\grad\mu^{n+1}|^2$,
$\tfrac1{d_0}|\cs\grad\mu^{n+1}-d_0\grad(\As q^{n+1})|^2$,
$\ks(q^{n+1})^2$ and $\epsq|\grad q^{n+1}|^2$
are precisely the four summands in the definition of $\Dmod^{n+1}$.
Hence
\begin{equation}\label{eq:app-IV}
\gamma\inner{\grad\phi^{n+1}}{\Dtwo\grad\phi^{n+1}}+2r^{n+1}\Dtwo r^{n+1}
+\inner{\Dtwo q^{n+1}}{q^{n+1}} = -\Dmod^{n+1} .
\end{equation}
This identity has no remainder depending on $\phi^{n+1}-\phi^\star$. A
naive linearisation would freeze only some of the coefficients and
evaluate the others at $\phi^{n+1}$. Here, every coefficient on the
right-hand side of \eqref{eq:app-III} is frozen (it carries the subscript
$\str$ and is independent of $\phi^{n+1}$ and $q^{n+1}$), so \eqref{eq:CS}
applies exactly.

\emph{Step 4: Telescoping the left-hand side.} Each of the three terms on
the left of \eqref{eq:app-IV} has the form
$\inner{\Dtwo a^{n+1}}{a^{n+1}}$. For $a=\grad\phi$ and $a=q$ this holds
directly. For $a=r$ it holds because $r$ is a scalar, so that
$2r^{n+1}\Dtwo r^{n+1}=2\inner{\Dtwo r^{n+1}}{r^{n+1}}$. By the definition
\eqref{eq:bdf2-ops}, $2\dt\,\Dtwo a^{n+1}=3a^{n+1}-4a^n+a^{n-1}$. Taking
the inner product with $a^{n+1}$ and applying the G-stability identity
\eqref{eq:G} gives
\begin{equation}\label{eq:app-Gapp}
2\dt\inner{\Dtwo a^{n+1}}{a^{n+1}}=\inner{3a^{n+1}-4a^n+a^{n-1}}{a^{n+1}}
=\half\bigl[\Phimix(a^{n+1},a^n)-\Phimix(a^n,a^{n-1})\bigr]+\half\norm{\dd a^{n+1}}^2 .
\end{equation}
Applying \eqref{eq:app-Gapp} with $a=\grad\phi$ and multiply the whole line
by $\gamma$ to get
\begin{equation}\label{eq:app-Gapp1}
2\dt\,\gamma\inner{\grad\phi^{n+1}}{\Dtwo\grad\phi^{n+1}}
=\tfrac\gamma2\bigl[\Phimix(\grad\phi^{n+1},\grad\phi^n)-\Phimix(\grad\phi^n,\grad\phi^{n-1})\bigr]
+\tfrac\gamma2\norm{\grad\dd\phi^{n+1}}^2 .
\end{equation}
Applying \eqref{eq:app-Gapp} with $a=q$, we get
\begin{equation}\label{eq:app-Gapp2}
2\dt\inner{q^{n+1}}{\Dtwo q^{n+1}}
=\half\bigl[\Phimix(q^{n+1},q^n)-\Phimix(q^n,q^{n-1})\bigr]+\half\norm{\dd q^{n+1}}^2 .
\end{equation}
Applying \eqref{eq:app-Gapp} with $a=r$ and then multiply the  by $2$
(matching the factor of $2$ already present in $2r^{n+1}\Dtwo r^{n+1}$) gives
\begin{equation}\label{eq:app-Gapp3}
2\dt\cdot2r^{n+1}\Dtwo r^{n+1}
=\bigl[\Phimix(r^{n+1},r^n)-\Phimix(r^n,r^{n-1})\bigr]+\bigl(\dd r^{n+1}\bigr)^2 .
\end{equation}
Adding \eqref{eq:app-Gapp1},\eqref{eq:app-Gapp2} and \eqref{eq:app-Gapp3}, the left-hand side is exactly $2\dt$
times the left-hand side of \eqref{eq:app-IV}, which by \eqref{eq:app-IV}
equals $2\dt\,(-\Dmod^{n+1})$, so
\begin{multline*}
2\dt\,\bigl(-\Dmod^{n+1}\bigr)
=\half\Bigl\{\gamma\bigl[\Phimix(\grad\phi^{n+1},\grad\phi^n)-\Phimix(\grad\phi^n,\grad\phi^{n-1})\bigr]
+\bigl[\Phimix(q^{n+1},q^n)-\Phimix(q^n,q^{n-1})\bigr]\\
+2\bigl[\Phimix(r^{n+1},r^n)-\Phimix(r^n,r^{n-1})\bigr]\Bigr\}
+\tfrac\gamma2\norm{\grad\dd\phi^{n+1}}^2+\half\norm{\dd q^{n+1}}^2+\bigl(\dd r^{n+1}\bigr)^2 .
\end{multline*}
The curly-braced quantity is exactly $4(\Emod^{n+1}-\Emod^n)$ by the
definition \eqref{eq:Emod} of $\Emod^n$,
\begin{multline*}
\Emod^{n+1}-\Emod^n=\tfrac14\Bigl\{\gamma\bigl[\Phimix(\grad\phi^{n+1},\grad\phi^n)-\Phimix(\grad\phi^n,\grad\phi^{n-1})\bigr]\\
+\bigl[\Phimix(q^{n+1},q^n)-\Phimix(q^n,q^{n-1})\bigr]
+2\bigl[\Phimix(r^{n+1},r^n)-\Phimix(r^n,r^{n-1})\bigr]\Bigr\} ,
\end{multline*}
the constant $-\CO$ appearing identically inside $\Emod^{n+1}$ and
$\Emod^n$ having cancelled upon subtraction and
\begin{multline*}
2\dt\,\bigl(-\Dmod^{n+1}\bigr)
=\half\cdot4\bigl(\Emod^{n+1}-\Emod^n\bigr)+\tfrac\gamma2\norm{\grad\dd\phi^{n+1}}^2
+\half\norm{\dd q^{n+1}}^2+\bigl(\dd r^{n+1}\bigr)^2\\
=2\bigl(\Emod^{n+1}-\Emod^n\bigr)+\tfrac\gamma2\norm{\grad\dd\phi^{n+1}}^2
+\half\norm{\dd q^{n+1}}^2+\bigl(\dd r^{n+1}\bigr)^2 .
\end{multline*}
Dividing this last equation through by $2$ and rearranging gives
\[
\bigl(\Emod^{n+1}-\Emod^n\bigr)
+\tfrac14\bigl[\gamma\norm{\grad\dd\phi^{n+1}}^2+\norm{\dd q^{n+1}}^2
+2\bigl(\dd r^{n+1}\bigr)^2\bigr] = -\dt\,\Dmod^{n+1} ,
\]
which, after moving $-\dt\Dmod^{n+1}$ to the left, is exactly \eqref{eq:star}.

\emph{Step 5: Non-negativity.} Each term in the definition
\eqref{eq:Dmod} of $\Dmod^{n+1}$ is the integral of a pointwise
non-negative function:
\begin{itemize}
\item $\bigl(\bs-\tfrac{\cs^2}{d_0}\bigr)|\grad\mu^{n+1}|^2\ge0$ by
\Cref{ass:coerc};
\item $\tfrac1{d_0}|\cs\grad\mu^{n+1}-d_0\grad(\As q^{n+1})|^2\ge0$, since
it is a square and $d_0>0$;
\item $\ks(q^{n+1})^2\ge0$, since $\ks\ge0$;
\item $\epsq|\grad q^{n+1}|^2\ge0$, since $\epsq\ge0$.
\end{itemize}
Hence $\Dmod^{n+1}\ge0$, which completes the proof.

\paragraph{Proof of the filtered-scheme stability extension (\Cref{rmk:filter}).}
When coefficients with strong curvature, such as $\As$ (which behaves like
$\cot(\pi\phi)$), are evaluated pointwise on a spectral grid, products of
fields can generate spurious high-frequency (aliasing) noise that has no
counterpart in the continuous solution. For the long-time runs of
\Cref{sec:physics}, we suppress this noise by applying the $2/3$-rule
filter $F$ after every step. The filter discards the highest third of the
resolved wavenumbers in each direction.

Filtering changes the fields that are stored and used in the next step.
Therefore, \Cref{thm:main}, which was proved for the unfiltered scheme,
does not apply directly to the scheme that is actually run. We show here
that the identity \eqref{eq:star} extends to the filtered scheme, with two
additional non-negative terms. Consequently, the scheme that produced the
results of \Cref{sec:physics} is provably stable.

\medskip
\noindent\emph{Properties of the filter.}
In Fourier space, $F$ multiplies each mode by $1$ (kept) or $0$
(discarded), depending only on the wavenumber. Hence $F$ is diagonal in
the Fourier basis, which gives three properties:
\begin{itemize}
\item $F^2=F$, since multiplying a mode by $0$ or $1$ twice is the same
as doing so once;
\item $F^\ast=F$, since $F$ is a real diagonal operator in an orthonormal
basis;
\item $\grad F=F\grad$, since $\grad$ is also diagonal in the Fourier
basis, and diagonal operators commute.
\end{itemize}
It follows that $F(I-F)=F-F^2=0$, and therefore
$\inner{Fu}{(I-F)u}=\inner{u}{F(I-F)u}=0$ for every field $u$. In
particular,
\begin{equation}\label{eq:F-orth}
\norm u^2=\norm{Fu}^2+\norm{(I-F)u}^2 .
\end{equation}

\medskip
\noindent\emph{Effect of the filter on $\Phimix$.}
Let $v$ be a field with $Fv=v$, and let $u$ be arbitrary. We claim that
\begin{equation}\label{eq:Phi-filter}
\Phimix(u,v)-\Phimix(Fu,v)=5\norm{(I-F)u}^2 .
\end{equation}
The first term of $\Phimix$ is handled by \eqref{eq:F-orth}:
$\norm u^2-\norm{Fu}^2=\norm{(I-F)u}^2$. For the second term, write
$2u-v=(2Fu-v)+2(I-F)u$. Since $Fv=v$, the field $2Fu-v$ lies in the range
of $F$ and is therefore orthogonal to $(I-F)u$. Hence
\[
\norm{2u-v}^2=\norm{2Fu-v}^2+4\norm{(I-F)u}^2 .
\]
Adding the two differences gives
$\norm{(I-F)u}^2+4\norm{(I-F)u}^2=5\norm{(I-F)u}^2$, which proves
\eqref{eq:Phi-filter}. Thus, filtering can only decrease $\Phimix$, and
the decrease is exactly five times the squared norm of the removed part.

\medskip
\noindent\emph{Application to the scheme.}
At each step, the stored previous levels are outputs of the filter, so
the hypothesis $Fv=v$ holds for them (by induction over the steps, with
filtered initial data). We apply \eqref{eq:Phi-filter} with
\begin{itemize}
\item $u=\grad\phi^{n+1}$ and $v=\grad\phi^n$, where $\phi^n$ is the
stored, already filtered field (note that $\grad F=F\grad$);
\item $u=q^{n+1}$ and $v=q^n$;
\item $u=r^{n+1}$ and $v=r^n$. Here no filtering is applied to the scalar
$r$, so this pair contributes nothing.
\end{itemize}
Replacing the raw fields $\phi^{n+1},q^{n+1}$ by their filtered versions
$F\phi^{n+1},Fq^{n+1}$ in $\Emod^{n+1}$ lowers it by
\[
\tfrac{5\gamma}{4}\norm{(I-F)\grad\phi^{n+1}}^2
+\tfrac54\norm{(I-F)q^{n+1}}^2\ \ge\ 0 .
\]
Therefore \eqref{eq:star} holds for the filtered scheme with these two
non-negative terms added on the dissipative side. In particular,
\Cref{cor:stab} remains valid: the modified energy of the stored
(filtered) fields does not increase, for every $\dt>0$.

\paragraph{Proof of \Cref{thm:solv}.}
After spatial discretisation, \eqref{eq:S1}--\eqref{eq:S4} form a square
linear system for the unknowns $(\phi^{n+1},\mu^{n+1},q^{n+1},r^{n+1})$,
whose right-hand side depends on the data from the two previous time
levels. A square linear system $A\mathbf x=\mathbf b$ has exactly one
solution for every $\mathbf b$ if and only if the homogeneous system
$A\mathbf x=\mathbf 0$ has only the trivial solution $\mathbf x=\mathbf 0$,
that is, if the kernel of $A$ is trivial. The matrix $A$ is built from the
frozen coefficients $\bs,\cs,\As,\ks,\Gs$ and does not depend on the
previous-level data. It therefore suffices to consider the homogeneous
problem: we show that the scheme with zero previous-level data,
$\phi^n=\phi^{n-1}=q^n=q^{n-1}=r^n=r^{n-1}=0$, has only the trivial
solution. This is a property of the linear map defining the scheme, and
it implies unique solvability for arbitrary (non-zero) data in an actual
simulation.

To distinguish this homogeneous problem from an actual trajectory, we
denote its solution by $(\psi,\nu,\chi,\rho)$ instead of
$(\phi^{n+1},\mu^{n+1},q^{n+1},r^{n+1})$. For zero previous-level data,
the BDF2 operators \eqref{eq:bdf2-ops} reduce to
$\Dtwo\psi=3\psi/(2\dt)$ and $\dd\psi=\psi$, and similarly for $\chi$ and
$\rho$.

\medskip
\noindent\emph{Step 1: $\grad\psi=0$, $\chi=0$, $\rho=0$ and $\Dmod=0$.}
For zero previous-level data, $\Emod^n=-\CO$, and, using
$\Phimix(u,0)=5\norm u^2$,
\[
\Emod^{n+1}+\CO=\tfrac54\bigl[\gamma\norm{\grad\psi}^2+\norm\chi^2+2\rho^2\bigr].
\]
The identity \eqref{eq:star} holds for arbitrary previous-level data, so
it holds here as well. Inserting the expressions above (the remainder
term in \eqref{eq:star} contributes
$\tfrac14[\gamma\norm{\grad\psi}^2+\norm\chi^2+2\rho^2]$) gives
\[
\tfrac32\bigl[\gamma\norm{\grad\psi}^2+\norm\chi^2+2\rho^2\bigr]
+\dt\,\Dmod=0 ,
\]
where $\Dmod$ is \eqref{eq:Dmod} evaluated at $(\nu,\chi)$. Since
$\gamma>0$, and $\Dmod\ge0$ by \Cref{ass:coerc} (as in Step 5 of the
proof of \Cref{thm:main}), all terms are non-negative. Hence each of them
vanishes:
\[
\grad\psi=0,\qquad \chi=0,\qquad \rho=0,\qquad \Dmod=0 .
\]

\medskip
\noindent\emph{Step 2: $\psi=0$.}
Since $\grad\psi=0$, the function $\psi$ is constant in space. To
determine the constant, we test \eqref{eq:S1} with $1$. By
\eqref{eq:ibp1} with $g=1$, the right-hand side integrates to zero, so
$\int_\Om\Dtwo\psi=\tfrac{3}{2\dt}\int_\Om\psi=0$. Thus $\psi$ has zero
mean. A constant function with zero mean on $\Om$ vanishes, so $\psi=0$.

\medskip
\noindent\emph{Step 3: $\nu=0$.}
With $\chi=0$, the integrand of $\Dmod$ simplifies, because
$\bigl(\bs-\tfrac{\cs^2}{d_0}\bigr)+\tfrac{\cs^2}{d_0}=\bs$:
\[
\Dmod=\int_\Om\bigl(\bs-\tfrac{\cs^2}{d_0}\bigr)|\grad\nu|^2
+\tfrac1{d_0}\int_\Om|\cs\grad\nu|^2=\int_\Om\bs|\grad\nu|^2 .
\]
Since $\Dmod=0$ and $\bs\ge b_1>0$, we obtain $\grad\nu=0$, so $\nu$ is
constant in space. (This is the only place where the additional
assumption $\bs\ge b_1>0$ is used.) The constant is fixed by
\eqref{eq:S2}: with $\psi=0$ and $\rho=0$, it reads
$\nu=-\gamma\lap\psi+\rho\,\Gs=0$.

\medskip
\noindent We have shown $\psi=\nu=\chi=\rho=0$. Thus the homogeneous system
has only the trivial solution, and \eqref{eq:S1}--\eqref{eq:S4} has a
unique solution for every $\dt>0$ and for arbitrary data
$(\phi^n,q^n,r^n)$ and $(\phi^{n-1},q^{n-1},r^{n-1})$.

\section{The manufactured solution used in \Cref{sec:temporal}}\label{app:mms}

The temporal MMS check of \Cref{sec:temporal} uses the explicit, smooth
functions
\[
\phi_{\mathrm{ex}}(x,y,t) \eqdef 0.5 + 0.15\cos(2\pi x)\cos(2\pi y)\,e^{-3t},
\qquad
q_{\mathrm{ex}}(x,y,t) \eqdef 0.10\sin(2\pi x)\sin(2\pi y)\,e^{-2t},
\]
on $\Om=(0,1)^2$ at $\gamma=0.30$, chosen large enough that $\phi=0.5$ is
not spinodally unstable for this low wavenumber, so the forced problem
below is not numerically stiff. Neither function solves \eqref{eq:chq}.
With $\mu_{\mathrm{ex}}\eqdef-\gamma\lap\phi_{\mathrm{ex}}+f'(\phi_{\mathrm{ex}})$,
define the forcing pair
\begin{align*}
S_\phi(t) &\eqdef \pt\phi_{\mathrm{ex}}(t) - \dv\!\bigl(b(\phi_{\mathrm{ex}})\grad\mu_{\mathrm{ex}}
   -c(\phi_{\mathrm{ex}})\grad(A(\phi_{\mathrm{ex}})q_{\mathrm{ex}})\bigr), \\
S_q(t) &\eqdef \pt q_{\mathrm{ex}}(t) - \Bigl[-\kappa(\phi_{\mathrm{ex}})q_{\mathrm{ex}}
   -A(\phi_{\mathrm{ex}})\dv\!\bigl(c(\phi_{\mathrm{ex}})\grad\mu_{\mathrm{ex}}-d_0\grad(A(\phi_{\mathrm{ex}})q_{\mathrm{ex}})\bigr)
   +\epsq\lap q_{\mathrm{ex}}\Bigr],
\end{align*}
where $\pt\phi_{\mathrm{ex}},\pt q_{\mathrm{ex}}$ are the known, closed-form time
derivatives of the two functions above. By construction, $S_\phi,S_q$ are
exactly the residuals of \eqref{eq:chq-phi},\eqref{eq:chq-q} at
$(\phi_{\mathrm{ex}},q_{\mathrm{ex}})$. Adding $S_\phi,S_q$ to the right-hand sides
of \eqref{eq:chq-phi},\eqref{eq:chq-q} makes $(\phi_{\mathrm{ex}},q_{\mathrm{ex}})$ an
exact solution of the resulting, forced system. At the discrete level,
$S_\phi,S_q$ are evaluated at $t_{n+1}$ (using the same discrete spatial
operators as the scheme itself) and added to the right-hand sides of
\eqref{eq:S1},\eqref{eq:S3} only; \eqref{eq:S4} is left untouched, since
its correctness follows once $\phi^{n+1}$ already matches $\phi_{\mathrm{ex}}$.
The resulting forced scheme has a known exact solution at every $\dt$, so
the error $\|\phi^n-\phi_{\mathrm{ex}}(t_n)\|_\infty$ can be measured directly.
Because $(\phi_{\mathrm{ex}},q_{\mathrm{ex}})$ does not solve the unforced system,
this forced run is not expected to, and does not, satisfy the energy
identity \eqref{eq:star}. The check is purely about the truncation-error
order, not about stability.

\end{document}